\documentclass[12pt]{amsart}
\usepackage{amssymb,amsmath,xcolor,mathtools,ifthen}
\def\C {{\mathcal C}}
\def\L {{\mathcal L }}

\def\H {{\mathcal H}}

\def\M {{\mathcal M}}

\def\A {{\mathbb A}}

\def\CC {\mathbb{C}}
\def\R {\mathbb{R}}
\def\N {\mathbb{N}}

\def\eps{\varepsilon}
\def\e{{\rm e}}

\def\d{{\rm d}}

\def\i{{i}}
\def \l {\langle}
\def \r {\rangle}

\def \and {{\qquad\text{and}\qquad}}

\newtheorem{proposition}{Proposition}[section]
\newtheorem{theorem}[proposition]{Theorem}

\newtheorem{lemma}[proposition]{Lemma}
\theoremstyle{definition}
\newtheorem{definition}[proposition]{Definition}
\newtheorem{remark}[proposition]{Remark}

\numberwithin{equation}{section}
\newtheorem*{Acknowledgments}{Acknowledgments}

\DeclareMathOperator*{\dist}{dist}
\DeclareMathOperator{\re}{Re}
\DeclareMathOperator{\im}{Im}

\renewcommand{\ker}{\mathrm{Ker}}
\newcommand{\ran}{\mathrm{Ran}}
\newcommand{\Dom}{D}
\newcommand{\St}{(S(t))_{t\ge0}}
\newcommand{\Tt}{(T(t))_{t\ge0}}
\newcommand*{\Z}{{\mathbb{Z}}}
\newcommand*{\Lin}{{\mathcal{L}}}
\newcommand*{\abs} [1]{\lvert#1\rvert}
\newcommand*{\norm}[1]{\lVert#1\rVert}
\newcommand*{\set} [1]{\{#1\}}
\newcommand*{\setm}[2]{\{\,#1\mid#2\,\}}
\newcommand*{\iprod}[2]{\langle#1,#2\rangle}

\newcommand*{\Setm}[2]{\left\{\,#1\,\middle|\,#2\,\right\}}
\newcommand*{\Lp}[1][p]{L^{#1}}
\newcommand{\pmat}[1]{\begin{pmatrix}#1\end{pmatrix}}

\newcommand{\Abs}[2][default]{\ifthenelse{\equal{#1}{default}}{\left\lvert#2\right\rvert}{\ldelim{#1}{\lvert}#2\rdelim{#1}{\rvert}}}
\newcommand{\Norm}[2][default]{\ifthenelse{\equal{#1}{default}}{\left\lVert#2\right\rVert}{\ldelim{#1}{\lVert}#2\rdelim{#1}{\rVert}}}

\newcommand*{\Iprod}[3][default]{\ifthenelse{\equal{#1}{default}}{\left\langle#2,#3\right\rangle}{\ldelim{#1}{\langle}#2,#3\rdelim{#1}{\rangle}}}
\newcommand*{\Dualpair}[3][default]{\ifthenelse{\equal{#1}{default}}{\left\langle#2,#3\right\rangle}{\ldelim{#1}{\langle}#2,#3\rdelim{#1}{\rangle}}}

\newcommand{\eq}[1]{\begin{align*}#1\end{align*}}
\newcommand{\eqn}[1]{\begin{align}#1\end{align}}

\newcommand{\citel}[2]{\cite[#2]{#1}}

\newcommand{\gs}{\sigma}

\renewcommand{\gg}{\gamma}
\newcommand{\gd}{\delta}
\newcommand{\gl}{\lambda}

\newcommand{\inv}{^{-1}}

\newcommand*{\ddb}[2][1]{\ifthenelse{\equal{#1}{1}}{\frac{d}{d#2}}{\frac{d^{#1}}{d#2^{#1}}}}
\newcommand*{\pd}[3][1]{\ifthenelse{\equal{#1}{1}}{\frac{\partial{#2}}{\partial{#3}}}{\frac{\partial^{#1}{#2}}{\partial#3^{#1}}}}

\def \no#1#2#3 {{\bf #1} (#3), #2.}
\def \eds#1#2#3 {#1, #2, #3.}
\title[Wave-heat system with Coleman--Gurtin thermal law]
{Optimal decay for a wave-heat system\\ with Coleman--Gurtin thermal law}

\author[F. Dell'Oro, L. Paunonen, D. Seifert]
{Filippo Dell'Oro, Lassi Paunonen and David Seifert}

\address[F. Dell'Oro]{Dipartimento di Matematica, Politecnico di Milano
\newline\indent
Via Bonardi 9, 20133 Milano, Italy}
\email{filippo.delloro@polimi.it}

\address[L. Paunonen]{Mathematics and Statistics, Tampere University
\newline\indent PO.\ Box 692, 33101 Tampere, Finland}
\email{lassi.paunonen@tuni.fi}

\address[D. Seifert]{School of Mathematics, Statistics and Physics, Newcastle University
\newline\indent
Newcastle upon Tyne, NE1 7RU, United Kingdom}
\email{david.seifert@ncl.ac.uk}

\thanks{L. Paunonen was supported by the Academy of Finland under grants 298182 and 310489.}

\begin{document}

%%%%%%%%%%%%%%%%%%%%%%%%%%%%%%%%%%%%%%%%%%%%%%%%%
\begin{abstract}
We study the long-term behaviour of solutions to a one-dimensional coupled wave-heat system with Coleman--Gurtin thermal law. Our approach is based on the asymptotic theory of $C_0$-semigroups and recent results developed for coupled control systems. As our main results, we represent the system as a feedback interconnection between the wave part and the Coleman--Gurtin part and we show that the associated semigroup in the history framework of Dafermos is polynomially stable with optimal decay rate $t^{-2}$ as $t\to\infty$. In particular, we obtain a sharp estimate for the rate of energy decay of classical solutions to the problem.
\end{abstract}
%%%%%%%%%%%%%%%%%%%%%%%%%%%%%%%%%%%%%%%%%%%%%%%%%

\maketitle

\section{Introduction}

\noindent
The study of the asymptotic behaviour of solutions 
to coupled PDE systems has attracted a considerable amount of attention in the recent literature.
In this article, we focus on a one-dimensional coupled wave-heat system consisting 
of a wave equation and a Coleman--Gurtin equation.
More specifically, we use the asymptotic theory of strongly continuous semigroups combined with recent results on coupled abstract control systems to derive an optimal rational decay rate for classical solutions 
to the system
\begin{equation}
\label{WGP}
\left\{
\begin{aligned}
u_{tt}(x,t) &=u_{xx}(x,t),  &x \in (-1,0),\ t>0,\\
w_t(x,t) &= w_{xx}(x,t)+ \int_0^\infty g(s) w_{xx}(x,t-s)\, \d s,  &x \in (0,1),\ t>0.
\end{aligned}
\right.
\end{equation}
The equations are coupled, for $t>0$, through the transmission conditions
\begin{equation}
\label{WGPtransmission}
u_{t}(0,t) =w(0,t),\qquad
u_{x}(0,t) =w_{x}(0,t)+\int_0^\infty g(s) w_{x}(0,t-s) \,\d s
\end{equation}
at the interface $x=0$, and in addition we impose the Dirichlet boundary conditions
\begin{equation}
\label{WGPboundary}
u(-1,t) =w(1,t)=0.
\end{equation}
 The convolution kernel $g: [0,\infty)\to[0,\infty)$ is a convex integrable function (thus non-increasing and vanishing
at infinity) of unit total mass, taking the explicit form
$$
g(s)= \int_s^\infty \mu(r) \,\d r,\qquad s\ge0,
$$
where $\mu:(0,\infty)\to [0,\infty)$ is a non-increasing
absolutely continuous integrable function (possibly unbounded near zero).
In particular, $\mu$ is differentiable almost everywhere with $\mu'(s)\leq0$ for almost every $s>0$. 
Finally, we impose initial conditions of the form
\begin{equation*}
\label{WGPiv}
\left\{
\begin{aligned}
u(x,0) &=u_0(x),  &x \in (-1,0),\\
u_t(x,0) &=v_0(x),& x \in (-1,0),\\
w(x,0)&=w_0(x),\quad w(x,-s)=\varphi_0(x,s),\quad& x \in (0,1),\ s>0,
\end{aligned}
\right.
\end{equation*}
where $u_0,v_0,w_0,\varphi_0$ are assigned data.
In particular, $\varphi_0$ accounts for the so-called initial past history
of $w$. 

The stability analysis of coupled wave-heat systems has been the subject of intensive investigations
over the past few decades.
Their intrinsic mathematical interest apart, the main motivation for studying such systems stems from the fact that they can be viewed
as linearisations
of more complex fluid-structure models arising in fluid mechanics; see for instance \cite{avalos2007mathematical,rauch2005polynomial,MR2289863}.
In the absence of the integral term, \eqref{WGP} reduces
to the classical wave-heat system, whose asymptotic properties
have been extensively analysed in the literature; see for instance \cite{AvaLas16, AvaTri13, BatPau16, Duy07,MR4015695,MR2085542, MR2289863} 
and the references therein.
In particular, it is known that in this case the associated solution semigroup is \emph{semi-uniformly stable} in the sense that all classical solutions converge to zero at a uniform rate, and more specifically the semigroup is \emph{polynomially stable} with optimal decay rate $t^{-2}$ as $t\to\infty$. In particular, the semigroup fails to be exponentially stable. To the best of the authors' knowledge, the system in \eqref{WGP} with a non-trivial
kernel $g$ was first studied in \cite{Zha14}. In fact, the analysis in
\cite{Zha14} deals with a more general system in which the Laplacian $w_{xx}$ appearing in the second equation
is replaced by $\beta w_{xx}$ for some $\beta\geq0$.
The cases $\beta>0$ and $\beta=0$ correspond to the so-called Coleman--Gurtin \cite{MR214334}
and Gurtin--Pipkin \cite{MR1553521} models, respectively. Thus our system \eqref{WGP} corresponds to
the Coleman--Gurtin case with $\beta=1$, a choice which entails no essential loss of generality.
One of the main results of \cite{Zha14} is that if $\beta=0$ and if
 the so-called Dafermos condition
\begin{equation}
\label{conddaf}
\mu'(s) + \delta\mu(s)\leq0
\end{equation}
holds for some $\delta>0$ and almost every $s>0$, then the semigroup associated with the
wave-Gurtin--Pipkin
system in the history
space framework of Dafermos \cite{Daf70}
is exponentially stable. Since the Gurtin--Pipkin dissipation given solely by
the convolution term $\int_0^\infty g(s) w_{xx}(t-s)\, \d s$ is weaker than the dissipation provided
by the Laplacian $w_{xx}$, this result serves to illustrate that the classical wave-heat system
fails to be exponentially stable on account of \emph{overdamping}.
It is a reasonable guess, therefore,
that the wave-Coleman--Gurtin system \eqref{WGP}, too, fails to be exponentially stable, and this has been confirmed in
\cite{Zha14}, at least in the special case where $g$ is an exponential function.

In the present paper, we complete the
analysis begun in \cite{Zha14} by finding the optimal (semi-uniform)
decay rate of the semigroup $\St$ associated with \eqref{WGP} in the history
space framework of Dafermos. More precisely, assuming the condition
\begin{equation}
\label{condmu}
\mu(t+s)\leq C \e^{-\delta t} \mu(s)
\end{equation}
for some $C\geq1$ and $\delta>0$ and for every $t\geq0$ and $s>0$, we show in Theorem~\ref{thm:decay} and Proposition~\ref{prp:opt} that $\St$ is
polynomially stable with optimal decay rate $t^{-2}$ as $t\to\infty$.
Observe that this decay rate coincides with that of the classical wave-heat system.
Note also that \eqref{condmu} is weaker than \eqref{conddaf}. For instance, 
condition \eqref{conddaf}, in contrast to \eqref{condmu}, does not allow flat zones or horizontal inflection points; see for instance \cite{MR2294601, MR2215885}.
Our approach consists in writing the system in \eqref{WGP}--\eqref{WGPtransmission} as
a \emph{feedback interconnection} between the wave part and the Coleman--Gurtin part.
Based on this decomposition we show that the infinitesimal generator $\A$ of $\St$
has a special block operator structure
which can be used to derive a sharp resolvent estimate; see Theorems~\ref{thm:Resolvent} and \ref{stima basso}. This resolvent estimate, combined with the
asymptotic theory of $C_0$-semigroups, finally leads to the desired decay estimates.
The general decomposition approach used in this paper extends 
to the case where the wave part in~\eqref{WGP} has spatially varying parameters, 
and also to more complex systems, such as chains consisting of several coupled 
wave and Coleman--Gurtin-type equations.
In the latter case the decoupling approach reduces the study of the more complicated system 
to the analysis of its simpler constituent parts,
 and in this way facilitates efficient treatment of chains of coupled equations.
The same approach can potentially also be employed in the stability analysis of coupled PDEs on networks.

\subsection*{Notation}
We introduce the (complex) Hilbert spaces
\begin{align*}
H^1_l(-1,0) =\{\varphi \in H^1(-1,0) : \varphi(-1)=0 \}, \quad\,\,
H^1_r(0,1) =\{\varphi \in H^1(0,1) : \varphi(1)=0 \},
\end{align*}
with the inner products
$\l \varphi ,\psi\r_{H^1_l(-1,0)} = \l \varphi' ,\psi' \r_{L^2(-1,0)}$ and
$\l \varphi ,\psi\r_{H^1_r(0,1)}= \l \varphi' ,\psi' \r_{L^2(0,1)}$.
We also introduce the so-called memory space $\M = L^2_\mu(0,\infty; H^1_r(0,1))$
of  $H^1_r(0,1)$-valued functions on $(0,\infty)$ which are square-integrable with respect to the measure $\mu(s)\d s$,
endowed with the natural inner product
$$\langle\eta,\xi\rangle_\M=\int_0^\infty \mu(s)\langle\eta(s),\xi(s)\rangle_{H^1_r(0,1)}\, \d s.
$$
The state space of our problem will be
$$
\H = H^1_l(-1,0) \times L^2(-1,0) \times L^2(0,1) \times \M,
$$
with the natural inner product
$$
\l (u,v,w,\eta) , (\tilde u, \tilde v, \tilde w ,\tilde \eta) \r_\H =  \l u, \tilde u \r_{H^1_l(-1,0)} +
\l v, \tilde v\r_{L^2(-1,0)} +\l w, \tilde w\r_{L^2(0,1)} + \l \eta,\tilde \eta \r_\M.
$$

Throughout the paper, the Young, H\"older and Poincar\'e inequalities
will be used without explicit mention.
Square roots of complex numbers are defined with a branch cut along $(-\infty,0]$. In particular,
$\re\sqrt{\lambda}\ge0$ for all $\lambda\in\CC$, with strict inequality for $\lambda\not\in(-\infty,0]$.
We denote the open right and left half-planes in the complex plane by $\CC_\pm=\{\lambda\in\CC:\re\lambda\gtrless0\}$. Given (complex) Banach spaces $X$ and $Y$ we write $\L(X,Y)$ for the space of bounded linear operators from $X$ to $Y$, and we write $\L(X)$ instead of $\L(X,X)$. If $A$ is a closed linear operator acting on a Banach or Hilbert space, we denote its spectrum by $\gs(A)$ and its resolvent set by $\rho(A)$. 
We frequently consider the domain $D(A)$ of $A$ as being endowed with the graph norm $\norm{x}_{A}=(\norm{x}^2+\norm{Ax}^2)^{1/2}$. In particular, $\Dom(A)$ is a Hilbert space whenever $X$ is.
Moreover, for $\lambda\in\rho(A)$ we write $R(\lambda,A)$ for the resolvent operator $(\lambda-A)^{-1}$. Finally, we use conventional asymptotic notation, including `big O' and `little o', and we occasionally write $p\lesssim q$ to indicate that $p\le Cq$ for some (implicit) constant $C>0$.

%%%%%%%%%%%%%%%%%%%%%%%%%%%%%%%%%%%%%%%%%%%%%%%%%
\section{The System Operator and Wellposedness}
\label{opa}

\noindent
We begin by introducing the infinitesimal generator of the right-translation semigroup on $\M$, that is,
the linear operator
$$T\eta=-\eta_s,\qquad D(T)=\big\{\eta\in{\M}:\eta_s\in\M,\
\lim_{s\to 0}\|\eta(s)\|_{H^1_r(0,1)}=0\big\},$$
where $\eta_s$ denotes the (weak) derivative of $\eta$ with respect to the variable $s>0$.
Integration by parts with respect to $s$ together with a limiting argument can be used to show (as in \cite{GraPat02}) that
\begin{equation}
\label{Teta}
\re \l T\eta,\eta\r_\M
= \frac12 \int_0^\infty \mu'(s) \|\eta(s)\|_{H^1_r(0,1)}^2\, \d s\leq 0,\qquad \eta \in D(T).
\end{equation}
With a view to rewriting \eqref{WGP}--\eqref{WGPtransmission} in the \emph{history
space framework} of Dafermos~\cite{Daf70}, we consider for each $t>0$ the auxiliary function
$$
\eta^t(x,s)= \int_0^s w(x,t-\sigma) \,\d \sigma ,\qquad x \in (0,1),\ s>0,
$$
accounting for the integrated past history of $w$.
We further introduce, still in the spirit of \cite{Zha14}, the function
\begin{equation}
\label{PHIDEF}
\phi(x,t) = w(x,t) + \int_0^\infty \mu(s) \eta^t(x,s) \,\d s,\qquad x\in(0,1),\ t>0.
\end{equation}
Integrating by parts (formally) we obtain the identity
$$
w(x,t)+\int_0^\infty g(s) w(x,t-s) \,\d s =\phi(x,t),\qquad x\in(0,1),\ t>0.
$$
The system \eqref{WGP}--\eqref{WGPtransmission} can now be rewritten as
\begin{equation}
\label{WGPdaf}
\left\{
\begin{aligned}
u_{tt}(x,t) &=u_{xx}(x,t), & x \in (-1,0),\ t>0,\\
w_t(x,t) &=  \phi_{xx}(x,t),& x \in (0,1),\ t>0,\\
\eta^t_t(x,s)&= T \eta^t(x,s) + w(x,t),\quad & x \in (0,1),\ s,t>0,
\end{aligned}
\right.
\end{equation}
with the boundary conditions \eqref{WGPboundary} and the coupling conditions
\eqn{
\label{WGPdafBC}
u_{t}(0,t) =w(0,t),\qquad
u_{x}(0,t) = \phi_{x}(0,t)
}
for $t>0$.
By introducing the state vector $z(t)= (u(\cdot,t),v(\cdot,t),w(\cdot,t),\eta^t(\cdot,\cdot))^T$,
we may convert the above problem into an abstract Cauchy problem in the space $\H$, namely
\begin{equation}
\label{eq:ACP}
\left\{
\begin{aligned}
\dot{z}(t) &= \A z(t),\qquad t\ge0,\\
z(0)&=z_0,
\end{aligned}\right.
\end{equation}
where $\A: D(\A) \subset \H \to \H$ is the linear operator
$$
\A
\left(\begin{matrix}
u\\
v\\
w\\
\eta
\end{matrix}
\right)
=\left(
\begin{matrix}
v\\
u''\\
\phi''\\
T\eta + w
\end{matrix}
\right),\qquad
D(\A) =\Setm{ \left(\begin{matrix}
u\\
v\\
w\\
\eta
\end{matrix}
\right) \in \H}{
\begin{matrix}
u \in H^2(-1,0)\\
v \in H^1_l(-1,0)\\
w \in H^1_r(0,1)\\
\eta \in \Dom(T)\\
\phi \in H^2(0,1)\\
u'(0)= \phi'(0)\\
v(0)=w(0)
\end{matrix}
},
$$
and  $z_0=(u_0,v_0,w_0,\eta^0)^T\in\H$ with  $\eta^0(x,s) = \int_0^s \varphi_0(x,\sigma)\, \d \sigma$ for $x\in(0,1)$ and $s>0$.
Theorem~\ref{thm:SG} below shows that $\A$ generates a contraction semigroup on $\H$. The proof is based on the special block operator structure of $\A$ introduced in the same result. This structure of $\A$ also plays  a central role later in Section~\ref{sec:ResolventEstimates}, where we use it together with the results in~\cite{Pau19} in order to derive an optimal resolvent estimate for $\A$ on $i\R$.

To state the theorem, we first define some notation related to extrapolation spaces for semigroup generators.
If $A: \Dom(A)\subset Z\to Z$ generates a $C_0$-semigroup $\Tt$ on a Hilbert space $Z$, then $\Dom(A)$ is a Hilbert space with respect to the graph norm of $A$. We define $Z_{-1}$ to be the completion of the space $Z$ with respect to the norm $\norm{z}_{Z_{-1}}=\norm{(\gl_0-A)\inv z}_Z$ with $\gl_0\in\rho(A)$ (the space $Z_{-1}$ is independent of the choice of $\gl_0\in\rho(A)$). The operator $A: \Dom(A)\subset Z\to Z$ extends to $A_{-1}: \Dom(A_{-1})\subset Z_{-1}\to Z_{-1}$ with domain $\Dom(A_{-1})=Z$; see for instance \citel{EngNag00book}{Sec.~II.5}. The operator $A_{-1}$ generates a  $C_0$-semigroup $(T_{-1}(t))_{t\ge0}$ on the Banach space $Z_{-1}$ such that for every $t\geq 0$ the operator $T_{-1}(t)\in \Lin(Z_{-1})$ is an extension of $T(t)\in \Lin(Z)$.
Finally, for an operator $B\in \Lin(\CC^m,Z_{-1})$ we let
 $$Z^B= \Dom(A)+ \ran(R(\gl_0,A_{-1})B)$$  
 for $\gl_0\in\rho(A)$ (the space $Z^B\subset Z$ is again is independent of the choice of $\gl_0\in\rho(A)$).

\begin{definition}[{\textup{\citel{TucWei14}{Def.~5.1}}}]
Assume that $A: \Dom(A)\subset Z\to Z$ generates a $C_0$-semigroup on $Z$ and that $C\in \Lin(\Dom(A),\CC^m)$.
The \textit{$\Lambda$-extension} of $C$ is defined as the operator
\eq{
C_\Lambda z = \lim_{\substack{\gl\to\infty\\\gl>0}} \gl CR(\gl,A) z
}
and the domain $\Dom(C_\Lambda)$ consists of those $z\in Z$ for which the limit exists.
\end{definition}

\begin{theorem}
  \label{thm:SG}
Let $Z_1=H_l^1(-1,0)\times L^2(-1,0)$ and $Z_2 = L^2(0,1)\times \M$.
There exist semigroup generators $A_k: \Dom(A_k)\subset Z_k\to Z_k$
 and operators
 $B_k\in \Lin(\CC,Z_{k,-1})$, $C_k\in \Lin(\Dom(A_k),\CC)$ for $k=1,2$, and a constant $D_1>0$ such that
\begin{subequations}
\label{eq:Ablockstruct}
\eqn{
  \A &= \pmat{A_{1,-1}&B_1C_{2\Lambda}\\-B_2C_{1\Lambda}&A_{2,-1}-B_2D_1C_{2\Lambda}},\\
  \Dom(\A)&= \Setm{\pmat{z_1\\z_2}\in Z_1^{B_1}\times Z_2^{B_2}}{ \begin{matrix}
      A_{1,-1}z_1+B_1 C_{2\Lambda}z_2\in Z_1\\
      A_{2,-1}z_2
      -B_2 (C_{1\Lambda}z_1+D_1 C_{2\Lambda}z_2) \in Z_2
\end{matrix}
}.
}
\end{subequations}
Moreover, the operator $\A$ generates a contraction semigroup $\St$ on $\H = Z_1\times Z_2$.
\end{theorem}

The proof of Theorem~\ref{thm:SG} is a direct consequence of Proposition~\ref{prp:Ablockstructure} at the end of this section.
Our approach in the proof of this result and subsequent ones does not require us to derive explicit expressions for the operators $B_1$, $B_2$ or the $\Lambda$-extensions $C_{1\Lambda}$ and $C_{2\Lambda}$ of the operators $C_1$ and $C_2$. 
In particular, explicit knowledge of these operators is not required for the purposes of proving well-posedness or deriving resolvent estimates for $\A$ by means of the results in~\cite{Pau19}.

\begin{remark}
The fact that $\A$ generates a contraction semigroup was already proved in~\cite{Zha14} under slightly stronger
assumptions on the memory kernel; cf.\ hypotheses (H1)-(H2) in~\cite{Zha14}. We also stress that \eqref{condmu}
is not needed in the semigroup generation part, but only in the resolvent estimates carried out in the next section. 
\end{remark}

Before proceeding to prove Theorem~\ref{thm:SG}, we shall motivate the block operator structure of $\A$ based on the properties of
the coupled PDE system~\eqref{WGPdaf} with the boundary conditions~\eqref{WGPboundary} and the coupling conditions~\eqref{WGPdafBC}.
The block structure in~\eqref{eq:Ablockstruct} arises from the decomposition of
the full coupled PDE system into two natural subparts: a wave equation and a Coleman--Gurtin-type diffusion equation.
Indeed,
if we introduce two auxiliary functions $U_1$ and $Y_1$,
the `wave part' of the coupled PDE system  is given by
\eqn{
  \label{eq:WavePDE}
  \left\{
  \begin{aligned}
    u_{tt}(x,t) &=u_{xx}(x,t), & x\in(-1,0),\ t>0,\\
    u(-1,t)&=0,\quad    u_t(0,t)=U_1(t), \quad  Y_1(t) = u_x(0,t),& t>0,\\
u(x,0)&=u_0(x), \quad u_t(x,0)=v_0(x), \quad& x\in (-1,0).
  \end{aligned}
  \right.
}
Thus, for $t>0$, the value for $u_t(\cdot,t)$ at $x=0$ is given by $U_1(t)$, while $Y_1(t)$ is determined by the value of $u_x(\cdot,t)$ at $x=0$.

Introducing two further auxiliary functions $U_2$ and $Y_2$, the remaining `Coleman--Gurtin part' is given by
\eqn{
  \label{eq:CGPDE}
\left\{
\begin{aligned}
    w_t(x,t) &=  \phi_{xx}(x,t), & x\in(0,1),\ t>0,\\
    \eta^t_t(x,s)&= T \eta^t(x,s) + w(x,t),& x\in(0,1),\ s,t>0,\\
    -\phi_x(0,t)&=U_2(t), \quad     Y_2(t) = w(0,t),\quad w(1,t)=0,&t>0,\\
w(x,0)&=w_0(x), \quad \eta^0(x,s) = \textstyle{\int_0^s \varphi_0(x,\gs)\,\d\gs}  , &\qquad x\in (0,1),\ s>0.
  \end{aligned}
  \right.
}
For $t>0$, the value of $\phi_x(\cdot,t)$ at $x=0$ is determined
by $U_2(t)$, and $Y_2(t)$ is determined by the value of $w(\cdot,t)$ at $x=0$.

The PDE models~\eqref{eq:WavePDE} and~\eqref{eq:CGPDE} become equivalent to the coupled PDE system~\eqref{WGPdaf} once we require
that for all $t> 0$ the auxiliary functions $U_1(t)$, $U_2(t)$, $Y_1(t)$, and $Y_2(t)$,
satisfy the identities
\eqn{
\label{eq:Feedback}\left\{
\begin{aligned}
  U_1(t)&=Y_2(t)\\
  U_2(t)&=-Y_1(t)
\end{aligned}
\right.
 \qquad \iff \qquad
\left\{
\begin{aligned}
u_{t}(0,t) &=w(0,t)\\
u_{x}(0,t) &= \phi_{x}(0,t),
\end{aligned}
\right.
}
which are precisely the coupling conditions~\eqref{WGPdafBC}.
The block operator structure~\eqref{eq:Ablockstruct} follows this decomposition of the coupled PDE into two parts.
In particular, the
operators $(A_1,B_1,C_1,D_1)$ are related to the wave part~\eqref{eq:WavePDE} and $(A_2,B_2,C_2)$ are related to the Coleman--Gurtin part~\eqref{eq:CGPDE}.
This decomposition
is moreover closely connected to
mathematical systems theory, where  $U_1$ and $U_2$ would be interpreted as the \emph{inputs} of the PDE models~\eqref{eq:WavePDE} and~\eqref{eq:CGPDE}, respectively, and $Y_1$ and $Y_2$ would define the \emph{outputs} of the two systems~\cite{CurZwa95book,Sta02,MalSta06}.
In the terminology of systems theory, the coupling conditions~\eqref{eq:Feedback} on the inputs $U_1(t)$ and $U_2(t)$ and the outputs $Y_1(t)$ and $Y_2(t)$ in~\eqref{eq:Feedback} define a \emph{feedback interconnection} between the wave part and the Coleman--Gurtin part.

In the remaining part of this section, we shall use the results from infinite-dimensional systems theory in~\cite{Sta02,MalSta06,Wei94b} to prove the block operator representation~\eqref{eq:Ablockstruct} of $\A$.

\subsection{Background on regular tuples and boundary nodes}
\label{backgroud}

The operators appearing in~\eqref{eq:Ablockstruct} form ``regular tuples'' 
in the sense of Definition~\ref{def:RLS} below. 
Such operators are closely related to the theory 
\textit{regular linear systems}~\cite{Wei94b}, \citel{TucWei14}{Sec.~5}.

\begin{definition}
Assume that $A$ generates a $C_0$-semigroup $\Tt$ on a Hilbert space $Z$.
An operator $B\in \Lin(\CC^m,Z_{-1})$
 is  \textit{admissible}~\citel{TucWei14}{Rem.~3.3} with respect to $\Tt$ if there exists $\tau>0$ such that
\eq{
\int_0^\tau T_{-1}(t)Bu(t)\,\d t\in Z, \qquad  u\in \Lp[2](0,\tau;\CC^m).
}
Correspondingly, an operator
  $C\in \Lin(\Dom(A),\CC^m)$ is 
 \textit{admissible}~\citel{TucWei14}{Rem.~3.4} with respect to $\Tt$ if there exist $\tau,\kappa>0$ such that
\eq{
\int_0^\tau\norm{CT(t)z}_{\CC^m}^2\,\d t\leq \kappa \norm{z}_Z^2, \qquad z\in \Dom(A).
}
\end{definition}

\begin{definition}
\label{def:RLS}
Assume that $A: \Dom(A)\subset Z\to Z$ generates a $C_0$-semigroup $\Tt$ on a Hilbert space $Z$ and that $B\in \Lin(\CC^m,Z_{-1})$ and $C\in \Lin(D(A),\CC^m)$ are admissible with respect to $\Tt$.
Then the tuple $(A,B,C,D)$ is said to be \emph{regular} if $D\in\CC^{m\times m}$, 
 $\ran( R(\gl,A_{-1}) B)\subset \Dom(C_\Lambda)$ for some (or, equivalently, all) $\gl\in\rho(A)$ and
 $$\sup_{\re\gl\geq \gs} \norm{C_\Lambda R(\gl,A_{-1})B}_{\CC^m}<\infty$$ for some $\gs\geq 0$. The \emph{transfer function} $P$ of the regular tuple $(A,B,C,D)$ is defined by
\eq{
P(\gl) = C_\Lambda R(\gl,A_{-1})B + D, \qquad \gl\in\rho(A).
}
The regular tuple $(A,B,C,D)$ is called \textit{impedance passive} if
\eqn{
\label{eq:RLSpassive}
\re \iprod{A_{-1}z+BU}{z}_Z\leq \re \iprod{C_\Lambda z+DU}{U}_{\CC^m}
}
for all $U\in \CC^m$ and $z\in Z^B$  satisfying $A_{-1}z+BU\in Z$.
\end{definition}

Choosing $z\in \Dom(A)$ and $U=0\in\CC^m$ in~\eqref{eq:RLSpassive} shows
that the semigroup generated by $A$ in an impedance passive regular tuple $(A,B,C,D)$ is contractive.

Our aim is to relate the wave part~\eqref{eq:WavePDE} and 
the Coleman--Gurtin part~\eqref{eq:CGPDE} of our coupled PDE system to regular tuples $(A_1,B_1,C_1,D_1)$ and $(A_2,B_2,C_2,D_2)$, respectively.
We shall do this by first formulating both of these PDEs
as
\emph{abstract boundary control systems}~\cite{CheMor03, MalSta06, Sal87a} of the form
\eqn{
\label{eq:BCSgen}
	\left\{\begin{aligned}
	  \dot{z}(t) &= L z(t),  \quad&t\ge0,\\
	  G z(t) &= U(t), &t\ge0,\\
	  Y(t) &= K z(t), &t\ge0,\\
	  z(0)&=z_0
	\end{aligned}\right.
}
on a Hilbert space $Z$ with $L:\Dom(L)\subset Z\to Z$ and $K,G: \Dom(L)\subset Z\to \CC^m$. As is shown in Lemma~\ref{lem:BCStoRLS} below,
under suitable assumptions
the operators $A$, $B$, $C$ and $D$ of the regular tuples exist and can be expressed in terms of $L$, $G$, and $K$.
The benefit of using the  framework of abstract boundary control systems 
is that~\eqref{eq:BCSgen} has a form which closely resembles both the wave part~\eqref{eq:WavePDE} and the Coleman--Gurtin part~\eqref{eq:CGPDE} 
with suitable choices of a differential operator $L: \Dom(L)\subset Z\to Z$ and boundary trace operators $G,K: \Dom(L)\subset Z\to \CC$.
We call~\eqref{eq:BCSgen} a \emph{boundary control system} if the operator $L$, $G$ and $K$ 
form a \emph{boundary node} defined as below.

\begin{definition}
\label{def:BCS}
The triple $(G,L,K)$ in~\eqref{eq:BCSgen} is said to be an \emph{(internally well-posed) boundary node} 
on the Hilbert spaces $(\CC^m,Z,\CC^m)$ (or sometimes, for short, on $Z$)  
if the linear operators $L:\Dom(L)\subset Z\to Z$ and $G,K: \Dom(L)\subset Z\to \CC^m$ 
have the following properties:
\begin{itemize}
\item[\textup{(a)}] The restriction $L\vert_{\ker(G)}: \ker(G)\subset Z\to Z$ generates a $C_0$-semigroup on $Z$;
\item[\textup{(b)}] $G,K\in \Lin(\Dom(L),\CC^m)$;
\item[\textup{(c)}] $\ran(G)=\CC^m$.
\end{itemize}
The boundary node is \emph{impedance passive} if
$$\re \iprod{Lz}{z}_Z\leq \re \iprod{Gz}{Kz}_{\CC^m}, \qquad z\in \Dom(L).$$
The \emph{transfer function} $P:\CC_+\to \CC^{m\times m}$ of an impedance 
passive boundary node of the form~\eqref{eq:BCSgen}
is defined so that, for $\gl\in\CC_+$ and $U\in \CC^m$, $$P(\gl)U = Kz,$$ 
where $ z\in \Dom(L)$ satisfies $(\gl-L)z=0$ and $Gz=U$.
\end{definition}

\begin{remark}
Conditions (a) and (b) in Definition~\ref{def:BCS} imply that $\ker(G)$ is a complete finite-codimen\-sional subspace of $\Dom(L)$ (equipped with the graph norm of $L$).
This in particular implies that $\Dom(L)$ is a Hilbert space or, equivalently, that $L$ is a closed operator.
%Conditions (b) and (c) in Definition~\ref{def:BCS} imply that $\ker(G)$ is a closed finite-codimen\-sional subspace of $\Dom(L)$ (equipped with the graph norm of $L$).
%In particular, $\Dom(L)$ is a Hilbert space or, equivalently, $L$ is a closed operator.
Moreover, $D(L)$ is densely and continuously embedded in $Z$.
\end{remark}

\begin{remark}
In defining the transfer function of a boundary node, we do not distinguish between $P$ in Definition~\ref{def:BCS} and its analytic extensions to domains containing $\CC_+$.
The existence and uniqueness of the solution $z\in \Dom(L)$ of the `abstract boundary value problem' $(\gl-L)z=0$ and $Gz=U$ for any $U\in\CC^m$ and $\gl\in\CC_+$  follow from~\citel{CheMor03}{Thm.~2.9}.
\end{remark}

The next lemma collects results from~\cite{MalSta06,Sta02,TucWei14,Wei94b} to show 
how an impedance passive boundary node $(G,L,K)$ on a Hilbert space $Z$ gives rise to a
regular tuple $(A,B,C,D)$ on the same space.

\begin{lemma}
\label{lem:BCStoRLS}
Let $(G,L,K)$ be an impedance passive boundary node on the Hilbert spaces $(\CC^m,Z,\CC^m)$.
Assume that the transfer function $P$ of the boundary node satisfies
\eq{
\sup_{s\in\R}\,\norm{P(\gs+is)}<\infty
}
for some $\gs\geq 0$ and that $P(\lambda)$ converges to a limit as $\lambda\to\infty$ through the positive reals.
Then there exists an impedance passive regular tuple $(A,B,C,D)$ on $Z$ such that
\eq{
A&=L\vert_{\ker(G)}, \\
Lz&=A_{-1}z+BGz, \\
Kz &= C_\Lambda z +DGz,\\
D&=\lim_{\substack{\gl\to \infty\\\gl>0}} P(\gl)\in \CC^{m\times m}
}
for $z\in \Dom(L)$. Furthermore, $Z^B=\Dom(L)$, $\ran(B)\cap Z=\set{0}$ and $P$
coincides with the transfer function of the regular tuple $(A,B,C,D)$ on $\rho(A)\cap \CC_+$.
\end{lemma}

\begin{proof}
By~\citel{MalSta06}{Thm.~2.3 and Prop.~2.5}, the boundary node $(G,L,K)$
defines a `system node' $S_{\mathrm{node}}$ in the sense of~\citel{MalSta06}{Def.~2.1} or~\citel{Sta02}{Def.~2.1}.
By definition, the system node $S_{\mathrm{node}}$ is a linear block operator
\eq{
S_{\mathrm{node}} = \pmat{A\&B\\C\&D}: \Dom(S_{\mathrm{node}})\subset Z\times \CC^m\to Z\times \CC^m
}
with components   $C\&D:D(S_\mathrm{node})\subset Z\times\CC^m\to\CC^m$ and
$$A\&B:D(S_\mathrm{node})\subset Z\times\CC^m\to Z,\qquad A\&B \pmat{z\\U} = A_{-1}z+BU,\qquad \pmat{z\\U}\in \Dom(S_{\mathrm{node}}),$$ where  $A: \Dom(A)\subset Z\to Z$ is the generator of a $C_0$-semigroup on $Z$ and $B\in \Lin(\CC^m,Z_{-1})$.
The result~\citel{MalSta06}{Thm.~2.3(ii)}
in particular shows that $A=L|_{\ker(G)}$ and that the `control operator' $B\in \Lin(\CC^m,Z_{-1})$ of the system node $S_{\mathrm{node}}$ satisfies $Lz = A_{-1}z+BGz$ for $z\in \Dom(L)$. Moreover, by~\citel{MalSta06}{Thm.~2.3(v)} we have $Z\cap \ran(B)=\set{0}$ and, letting $\lambda_0\in\rho(A)\cap\CC_+$, 
$$\Dom(L)=\Dom(A)+ \ran(R(\gl_0,A_{-1})B)=Z^B,$$
 while  by~\citel{MalSta06}{Thm.~2.3(iv)} the `combined observation and feedthrough operator' $C\&D$ of $S_{\mathrm{node}}$ is given by 
$$C\&D \pmat{z\\U}=Kz$$ for all $z\in \Dom(L)$ satisfying $Gz=U$. This further implies that the `observation operator' $C\in \Lin(\Dom(A),\CC^m)$ of $S_{\mathrm{node}}$ satisfies $Cz=Kz$ for all $z\in \Dom(A)=\ker(G)$. Moreover, the transfer function $P_{\mathrm{node}}$ of the system node \citel{Sta02}{Def.~2.1} then has the form
\eqn{
\label{eq:Pnode}
P_{\mathrm{node}}(\gl)U = C\&D \pmat{R(\gl,A_{-1})BU\\U}
= KR(\gl,A_{-1})BU
}
for all $U\in\CC^m$ and $\gl\in\rho(A)\cap\CC_+$. But if we write $z=R(\gl,A_{-1})BU\in Z^B=\Dom(L)$ then~\citel{MalSta06}{Thm.~2.3(v)} implies that $Gz=GR(\gl,A_{-1})BU=U$, and thus 
$$(\gl-L)z=(\gl -A_{-1})z-BGz = (\gl -A_{-1})R(\gl,A_{-1})BU-BU = 0.$$ 
This shows that in fact $P_{\mathrm{node}}(\gl)U=Kz=P(\gl)U$ for all $U\in\CC^m$ and $\gl\in\rho(A)\cap\CC_+$. By~\citel{Sta02}{Thm.~4.2} the system node $S_{\mathrm{node}}$ is impedance passive if (and only if) 
$$\re \iprod{A_{-1}z+BU}{z}_Z\leq \re \Iprod{C\&D \pmat{z\\U}}{U}_{\CC^m}$$ 
for all $z\in \Dom(L)$ and $U\in\CC^m$ satisfying $Gz=U$. This property holds since for any  $z\in \Dom(L)$ and $U\in \CC^m$ such that $Gz=U$ we have $A_{-1}z+BU=A_{-1}z+BGz=Lz$ and $C\&D(z,U)^T=Kz$, and thus
\eqn{
\label{eq:nodepassive}
\re\iprod{A_{-1}z+BU}{z}_Z
=\re\iprod{Lz}{z}_Z
\leq \re \iprod{Kz}{Gz}_{\CC^m}
=\re \Iprod{C\&D \pmat{z\\U}}{U}_{\CC^m}
}
by impedance passivity of the boundary node.
Furthermore, our assumption that $P$ (and thus also $P_{\mathrm{node}}$) is uniformly bounded on a vertical line in $\overline{\CC_+}$ together with~\citel{Sta02}{Thm.~5.1} shows that $S_{\mathrm{node}}$ is well-posed in the sense of~\citel{Sta02}{Def.~2.1} (or~\citel{TucWei14}{Def.~4.4}). In particular, the operators $B\in \Lin(\CC^m,Z_{-1})$ and $C\in \Lin(\Dom(A),\CC^m)$ are admissible with respect to the semigroup generated by $A$~\citel{TucWei14}{Prop.~4.9}.

Our assumption that $P(\gl)$ converges to a well-defined limit
as $\gl\to\infty$ with $\gl>0$ together with~\citel{TucWei14}{Thm.~5.6} (or~\citel{Wei94b}{Thm.~5.8}) implies that the system node $S_{\mathrm{node}}$ is `regular' in the sense of~\citel{TucWei14}{Def.~5.2}.
If we define $D=\lim_{\gl\to \infty,\lambda>0}P(\gl)$, then~\citel{TucWei14}{Thm.~5.5} shows that  $\ran(R(\gl,A_{-1})B)\subset \Dom(C_\Lambda)$ and the transfer function $P$ has the form
\eq{
P(\gl)=C_\Lambda R(\gl,A_{-1})B+D, \qquad \gl\in\CC_+.
}
Thus $(A,B,C,D)$ is regular in the sense of Definition~\ref{def:RLS}.
Finally, let $z\in \Dom(L)=Z^B$ be arbitrary. Then there exist $z_0\in \Dom(A)=\ker(G)$, $\gl_0\in\rho(A)\cap\CC_+$ and $U\in\CC^m$ such that $z=z_0+R(\gl_0,A_{-1})BU$. By \citel{MalSta06}{Thm.~2.3(v)} we have
$GR(\gl_0,A_{-1})B=I$, and hence $Gz=U$.
Now a direct computation using  $Cz_0=Kz_0$ and~\eqref{eq:Pnode} shows that
\eq{
Kz &= Kz_0 + KR(\gl_0,A_{-1})BU
= Cz_0 + P(\gl_0)U
\\&= C_\Lambda z_0 + C_\Lambda R(\gl_0,A_{-1})BU + DU\\
&= C_\Lambda z + DGz;
}
see also \citel{Wei94b}{Rem.~4.11}. Since the same computation also shows that 
$$C\&D \pmat{z\\U}=Kz = C_\Lambda z+DU$$ for $z\in \Dom(L)$ and $U\in\CC^m$ satisfying $Gz=U$, the estimate in~\eqref{eq:nodepassive} implies that the regular tuple $(A,B,C,D)$ is impedance passive.
\end{proof}

\begin{remark}\label{rem:m}
In the study of our wave-heat system we shall  require only 
the case $m=1$ of the general framework set out above. 
However, as already mentioned in the Introduction, the same framework
%We point out, however, that the same framework
 with $m>1$ can 
be used in an analogous way to analyse more complicated coupled systems, 
such as for instance the wave-heat-wave system.
\end{remark}

\subsection{The wave-part}

We now show that the wave part~\eqref{eq:WavePDE} 
can be written in the form~\eqref{eq:BCSgen} for some operators $G_1$, $L_1$ and $K_1$ defining a boundary node,
and that this representation also defines a regular tuple $(A_1,B_1,C_1,D_1)$ via Lemma~\ref{lem:BCStoRLS}.
Boundary control systems and regular tuples associated with one-dimensional 
and multidimensional wave equations are rather well understood; 
see for instance~\citel{MalSta06}{Sec.~5}, \citel{TucWei14}{Ex.~5.8} as well as \cite{GuoZha07,KurZwa15,ZwaLeG10}.
To prove this property for the wave part, we 
begin by identifying the operators $L_1$, $G_1$, and $K_1$ of the boundary node $(G_1,L_1,K_1)$. We can write~\eqref{eq:WavePDE} as a first order system
$$\left\{\begin{aligned}
    \pmat{u_t(x,t)\\v_t(x,t)} &= \pmat{v(x,t)\\u_{xx}(x,t)}, &x\in(-1,0),\ t>0,\\
v(0,t)&=U_1(t),\quad       Y_1(t) = u_x(0,t), \quad
    u(-1,t)=0, & t>0,\\
u(x,0)&=u_0(x), \quad v(x,0)=v_0(x), & x\in (-1,0).
\end{aligned}\right.$$
If, for $t\ge0$, we consider $z_1(t)=(u(\cdot,t),v(\cdot,t))^T$ to be the state 
of an abstract differential equation of the form~\eqref{eq:BCSgen}
on the Hilbert space $Z_1 = H_l^1(-1,0)\times L^2(-1,0)$,
then natural choices for the operators $G_1$, $L_1$ and $K_1$ of the boundary node $(G_1,L_1,K_1)$ are
\eq{
L_1 \pmat{u\\v}& =\pmat{v\\u''},
 \qquad \Dom(L_1)= (H^2(-1,0)\cap H_l^1(-1,0))\times H_l^1(-1,0),\\
G_1 \pmat{u\\v}&=v(0) \qquad \mbox{and}
\qquad
K_1 \pmat{u\\v}=u'(0)
}
for all
$(u,v)^T\in\Dom(L_1)$. In particular, the boundary
condition at $x=-1$ is part of the definition of $\Dom(L_1)$, and the condition at $x=0$ is determined by $G_1$.

\begin{proposition}
  \label{prp:RLSWavePart}
The tuple $(G_1,L_1,K_1)$ is an impedance passive boundary node on $(\CC,Z_1,\CC)$
and defines an impedance passive regular tuple $(A_1,B_1,C_1,D_1)$.
In particular, $D_1=1\in\CC$ and
the operator
\eq{
  A_1 \pmat{u\\v} = \pmat{v\\u''}, \quad
  \Dom(A_1) = \Setm{ \pmat{u\\v} \in H^2(-1,0) \times H_l^1(-1,0)}{u(-1)=v(0)=0}
}
is skew-adjoint with compact resolvent.
The spectrum of $A_1$ consists of simple eigenvalues, namely $\gs(A_1)=\setm{ik\pi}{k\in\Z}$. Writing
$\setm{\psi_k}{k\in\Z}$ for the corresponding set of orthonormal eigenvectors, the operator $C_1$ satisfies $\abs{C_1\psi_k}=1 $ for all $k\in\Z$.
\end{proposition}

\begin{proof}
It is easy to show that the restriction $A_1 = L_1\vert_{\ker(G_1)}: \Dom(A_1)\subset Z_1\to Z_1$ with the above domain 
is skew-adjoint and has compact resolvent.
In particular, $A_1$ generates a unitary group on $Z_1$.
It is also straightforward to show that $G_1,K_1\in \Lin(\Dom(L_1),\CC)$, and certainly $\ran(G_1)=\CC$.
Thus $(G_1,L_1,K_1)$ is a boundary node on $(\CC,Z_1,\CC)$ in the sense of Definition~\ref{def:BCS}. 
If $z=(u,v)^T\in \Dom(L_1)$ then
using $v(-1)=0$ we readily see, using integration by parts, that
\eq{
  \re \iprod{L_1 z}{z}_{Z_1}
  &= \re \l v(0), {u'(0)}\r_{\CC}
  = \re\iprod{G_1 z}{K_1 z}_{\CC}.
}
Thus $(G_1,L_1,K_1)$ is impedance passive.

To show that the wave part also defines a regular tuple,
we compute the transfer function $P_1$ of the boundary node $(G_1,L_1,K_1)$.
By definition, if $\gl\in\CC_+$ then
 $P_1(\gl)=K_1 z$, where $z=(u,v)^T\in \Dom(L_1)$ is such that
\eq{
  \left\{\begin{aligned}
    (\gl-L_1)z&=0\\
    G_1 z&=1
  \end{aligned}\right.
  \qquad \iff \qquad
  \left\{\begin{aligned}
    \gl u(x)&=v(x),&x\in(-1,0),\\
    \gl v(x)&=u''(x),&x\in(-1,0),\\
    u(-1)&=0,\quad
    v(0)=1.
  \end{aligned}\right.
}
We have
 $$u(x) = \frac{\sinh(\gl (x+1))}{\gl \sinh(\gl)},$$ 
 and hence
$P_1(\gl) = u'(0)  = \coth(\gl).$
Since $\sup_{\tau\in\R} \abs{\coth(1+i\tau)}<\infty$
and $P_1(\gl)\to 1$ as $\gl\to\infty$ with $\gl>0$, Lemma~\ref{lem:BCStoRLS} shows 
that the wave part defines an impedance passive regular tuple $(A_1,B_1,C_1,D_1)$ 
on $Z_1$ and that $D_1=\lim_{\gl\to\infty}P_1(\gl)=1$.

The eigenvalues of $A_1$ are $ik\pi$ for $k\in\Z$, and the corresponding orthonormal eigenvectors are given by
\eq{
  \psi_0 = \pmat{x+1\\0}
  \qquad\mbox{and}\qquad  \psi_k  = \frac{1}{k\pi}\pmat{\sin(k\pi (x+1))\\ ik\pi \sin(k\pi(x+1))}\quad  \mbox{for $k\ne0$}.
}
Since
$\setm{\psi_k}{k\in\Z}\subset \Dom(A_1)= \ker(G_1)$, we have
  $C_1\psi_k = K_1 \psi_k =(-1)^k$ for all $k\in\Z$.
\end{proof}

\subsection{Coleman--Gurtin part}
\label{subsecCG}

As our next step we  show that the Coleman--Gurtin part, too, defines an impedance passive regular tuple.
Based on the structure~\eqref{eq:CGPDE} we may consider $z_2(t) = (w(\cdot,t),\eta^t(\cdot,\cdot))^T$ for $t\ge0$ to be the state of the boundary node on the Hilbert space $Z_2 = L^2(0,1)\times \M$, and we may choose the operators $L_2: \Dom(L_2)\subset Z_2\to Z_2$ and $G_2,L_2: \Dom(L_2)\subset Z_2\to \CC$ as
    \eq{
      L_2
\pmat{w\\\eta}&= \pmat{\phi''\\T\eta+w},\qquad \Dom(L_2)= \Setm{\pmat{w\\\eta}\in H_r^1(0,1)\times \Dom(T)}{\phi\in H^2(0,1)},\\
  G_2 \pmat{w\\ \eta} &= -\phi'(0)\qquad\mbox{and} \qquad
  K_2 \pmat{w\\ \eta} = w(0)
    }
for all $(w,\eta)^T\in \Dom(L_2)$.

\begin{proposition}
\label{prp:RLSCGPart}
The tuple $(G_2,L_2,K_2)$ is an impedance passive boundary node on $(\CC,Z_2,\CC)$ 
and defines an impedance passive regular tuple $(A_2,B_2,C_2,D_2)$.
In particular, $D_2=0$ and the transfer function $P_2$  of the regular tuple
is given by
  \eq{
    P_2(\gl)
    = \frac{\tanh \sqrt{\gl/\ell(\gl)}}{\sqrt{\gl \ell(\gl)}}, \qquad \gl\in\CC_+,
}
where $\ell: \overline{\CC_+}\setminus \set{0}\to \CC$ is defined by
$$
\ell(\gl) = 1 + \frac{1}{\lambda} \int_0^\infty \mu(s) (1-\e^{-\lambda s})\, \d s.
$$
\end{proposition}

The proof of Proposition~\ref{prp:RLSCGPart}
requires the following lemma.

\begin{lemma}
\label{lem:GPBCSEllipticProblem-bis}
The operator $A_2=L_2\vert_{\ker(G_2)}$ satisfies $\ran(I-A_2)=Z_2$.
\end{lemma}

\begin{proof}
We begin by showing that, for every $\hat\eta\in\M$, the function $\hat{\xi}$ defined by 
$$
\hat \xi(x,s)= \int_0^s \e^{- (s-\sigma) } \hat \eta (x,\sigma) \,\d \sigma,\qquad x \in (0,1),\ s>0,
$$
belongs to $\M$ and satisfies the estimate
$\|\hat \xi\|_\M \leq \|\hat \eta\|_\M.$
To this end,
we introduce the auxiliary function $K$ defined by
$$
K(s) = \int_0^s \e^{-(s-\sigma)}\sqrt{\mu(\sigma)} \|\hat \eta(\sigma)\|_{H^1_r(0,1)} \,\d \sigma,\qquad s>0.
$$
By Young's inequality for convolutions we have
$\|K\|_{L^2(0,\infty)} \leq \|\hat \eta\|_\M$.
Hence, using the monotonicity of $\mu$, we obtain
$\|\hat \xi \|_{\M} \leq \|K\|_{L^2(0,\infty)} \leq \|\hat \eta\|_\M,$
as desired.

Next, for an arbitrarily given $\hat{z}=(\hat{w},\hat{\eta})^T\in Z_2$, we consider the equation $(I-A_2)z=\hat z$
in the unknown $z=(w,\eta)^T\in \Dom(A_2)$. Componentwise, we get the system
\begin{equation}
\label{three}
\left\{\begin{aligned}
w(x) - \phi''(x) &= \hat w(x), &x\in(0,1),\\
\eta(x,s) - T\eta(x,s) - w(x) &= \hat \eta(x,s),\quad &x\in(0,1),\ s>0.
\end{aligned}\right.
\end{equation}
Integrating the second identity  and using $\eta(x,0)=0$, we find
\begin{equation}
\label{etasol}
\eta(x,s)= (1- \e^{-s}) w(x) + \hat \xi(x,s),\qquad x\in(0,1),\ s>0.
\end{equation}
From the definition of $\phi$ we infer that
$w(x) = \phi(x)/\ell(1) -\hat \varrho(x)$,
where
$$\hat \varrho(x)= \frac{1}{\ell(1)}\int_0^\infty\mu(s)\hat \xi(x,s)\, \d s,\qquad x\in(0,1).
$$
Substituting into the first equation in \eqref{three}, we arrive at
$$
 \frac{\phi(x)}{\ell(1)} - \phi''(x) = \hat w(x)
 +\hat \varrho(x),\qquad x\in(0,1).
$$
The general solution of this equation above with the boundary condition $\phi(1)=0$
(coming from the fact that $w(1)=\eta(1,s)=0$) can be written as
\begin{equation}
\label{phisol}
\phi(x) = b\sinh \bigg(\frac{1-x}{\sqrt{\ell(1)}}\bigg) -\Phi(x),\qquad x\in(0,1),
\end{equation}
where $b\in \CC$ and
$$
\Phi(x)= \sqrt{\ell(1)} \int_{x}^1 \sinh \bigg(\frac{r-x}{\sqrt{\ell(1)}}\bigg)
(\hat w(r) +\hat \varrho(r))\, \d r,\qquad x\in(0,1).
$$
Accordingly, we have
\begin{equation}
\label{wsol}
w(x)= \frac{b}{\ell(1)}\sinh \bigg(\frac{1-x}{\sqrt{\ell(1)}}\bigg) -\frac{\Phi(x)}{\ell(1)}
- \hat \varrho(x),\qquad x\in(0,1).
\end{equation}
We now claim that $\phi\in H^2(0,1)$ and $w\in H^1_r(0,1)$. By \eqref{phisol} and \eqref{wsol},  the claim follows provided that $\hat \varrho\in H^1_r(0,1)$.
But the latter is true, since 
$$
\|\hat \varrho\|_{H^1_r(0,1)} \leq \frac{\sqrt{\kappa}}{\ell(1)} \|\hat \xi\|_{\M}
\leq \frac{\sqrt{\kappa}}{\ell(1)} \|\hat \eta\|_{\M}$$
by our earlier estimate, where $\kappa = \int_0^\infty \mu(s) \,\d s$ denotes the total mass of $\mu$.
Next we show that the function $\eta$ given by \eqref{etasol}
belongs to $\M$. Since we already know that $\smash{\hat \xi\in\M}$, we only need to prove that the map $s\mapsto
(1- \e^{-s}) w$ lies in $ \M$. But this follows from the estimate
$$
\int_0^\infty \mu(s) |1- \e^{-s}|^2 \|w\|_{H^1_r(0,1)}^2\, \d s \leq \kappa \|w\|_{H^1_r(0,1)}^2.
$$
Since $\eta,\hat\eta,w\in\M$, we also have $\eta_s  = \hat \eta + w - \eta \in \M$. Finally,
by monotonicity of $\mu$,
$$
\|\eta(s)\|_{H^1_r(0,1)} \leq |1-\e^{-s}|\|w\|_{H^1_r(0,1)}
+ \frac{\e^{s}}{\mu(s)} \bigg(\int_0^s \mu(\sigma)\,\d \sigma\bigg)^\frac12 \|\hat \eta\|_\M \to 0
$$
as $s\to 0$, and we have thus proved that $\eta\in D(T)$.

It remains only to show that the constant $b$ in \eqref{phisol} may be chosen in such a way that $\phi'(0)=0$, but a straightforward calculation yields
$$
b = \bigg[\frac{1}{\sqrt{\ell(1)}}\cosh\bigg(\frac{1}{\sqrt{\ell(1)}}\bigg)\bigg]^{-1}
 \int_0^1 \cosh \bigg(\frac{r}{\sqrt{\ell(1)}}\bigg)
\big(\hat w(r) +\hat \varrho(r)\big)\, \d r,
$$ 
and this completes the proof.
\end{proof}

\begin{proof}[Proof of Proposition~\textup{\ref{prp:RLSCGPart}}]
We begin by showing that $(G_2,L_2,K_2)$ is 
an impedance passive boundary node in the sense of Definition~\ref{def:BCS}.
We note first  that $\ran(G_2)=\CC$.
If $z = (w,\eta)^T\in \Dom(L_2)$, then using $w(1)=0$ and \eqref{Teta} we readily get
  \eq{
    \re  \iprod{L_2 z}{z}_{Z_2}
    &=\re \l w(0),-\phi'(0)\r_\CC - \|w\|_{H^1_r(0,1)}^2 +\re \l T\eta,\eta\r_\M\\
&\leq \re \l w(0),-\phi'(0)\r_\CC
    = \re \iprod{G_2 z}{K_2 z}_\CC.
  }
This estimate already shows that $(G_2,L_2,K_2)$ is impedance passive if it is a boundary node in the sense of Definition~\ref{def:BCS}. Moreover, the same estimate shows that
 $\re \iprod{L_2z}{z}\leq 0$ for $z\in \ker(G_2)$, and thus the restriction $A_2 = L_2\vert_{\ker(G_2)}$ is dissipative.
By Lemma~\ref{lem:GPBCSEllipticProblem-bis} we also have $\ran(I-A_2)=Z_2$, and therefore $A_2$ generates a contraction semigroup
on $Z_2$ by the Lumer--Phillips theorem.
In order to prove that $(G_2,L_2,K_2)$ is an impedance passive boundary node it remains to verify that $G_2,K_2\in \Lin(\Dom(L_2),\CC)$.
Recall that the norm on $\Dom(L_2)$ is taken to be the graph norm of $L_2$, that is to say
\eq{
\norm{z}_{\Dom(L_2)}^2
=\norm{L_2z}_{Z_2}^2 + \norm{z}_{Z_2}^2
=\norm{\phi''}_{\Lp[2](0,1)}^2+\norm{T\eta + w}_{\mathcal{M}}^2 +
\norm{w}_{\Lp[2](0,1)}^2+\norm{\eta }_{\mathcal{M}}^2,
}
for $z=(w,\eta)^T\in \Dom(L_2)\subset Z_2$.
Note first  that
\eq{
\abs{G_2z}=|\phi'(0)|
\lesssim \|\phi'\|_{H^1(0,1)}
\lesssim  \|\phi'\|_{L^2(0,1)}
+ \|z\|_{\Dom(L_2)}.
}
By interpolation and an application of Young's inequality we have $$ \|\phi'\|_{L^2(0,1)}\lesssim   \|\phi\|_{L^2(0,1)} + \|z\|_{\Dom(L_2)},$$
and hence
$\abs{G_2z} \lesssim \|\phi\|_{L^2(0,1)} + \|z\|_{\Dom(L_2)}.$
In order to show that $G_2\in \Lin(\Dom(L_2),\mathbb{C})$, it remains to
control the term $\|\phi\|_{L^2(0,1)}$. To this end we observe that, by definition of $\phi$, 
\begin{align*}
\|\phi\|_{L^2(0,1)} \lesssim \|z\|_{\Dom(L_2)} + \int_0^\infty \mu(s)\|\eta(s)\|_{H^1_r(0,1)} \,\d s
\lesssim \|z\|_{\Dom(L_2)} + \|\eta\|_\M
\lesssim \|z\|_{\Dom(L_2)},
\end{align*}
and hence boundedness of $G_2$ follows.
In order to show that $K_2\in \Lin(\Dom(L_2),\CC)$ we first note that, since $w(1)=0$, we have
$\abs{K_2z}=|w(0)|  \lesssim \|w\|_{H^1_r(0,1)}$.
Next, the definition of $\phi$ implies
\begin{align*}
\|w\|_{H^1_r(0,1)}
\lesssim \| \phi'\|_{L^2(0,1)}  +  \|\eta\|_{\mathcal{M}}
\lesssim\| \phi'\|_{L^2(0,1)} + \|z\|_{\Dom(L_2)}.
\end{align*}
The term $\| \phi'\|_{L^2(0,1)}$ can be estimated as before, and the boundedness of $K_2$ follows.
Thus $(G_2,L_2,K_2)$ is a boundary node on $(\CC,Z_2,\CC)$ in the sense of Definition~\ref{def:BCS}.

The transfer function $P_2$ of the boundary node is defined, for $\gl\in\CC_+$, by $P_2(\gl)=K_2z$, where
$z=(w,\eta)^T\in \Dom(L_2)$ solves the problem
$(\gl-L_2)z=0$ and $G_2 z=1$.
Arguing as in the proof of Lemma \ref{lem:GPBCSEllipticProblem-bis}, the first component of $z$
can be written as 
\eq{
w(x)= \frac{\sinh\big(\sqrt{\gl/\ell(\gl)}(1-x)\big)}{\sqrt{\lambda \ell(\gl)}
\cosh \sqrt{\gl/\ell(\gl)}},\qquad x\in(0,1),}
and hence
  \eq{
    P_2(\gl)
=  w(0)
= \frac{\tanh\sqrt{\gl/\ell(\gl)}}{\sqrt{\gl \ell(\gl)}}, \qquad \gl\in\CC_+.
}

We now show that
$P_2(\lambda)\to0$ uniformly in $\im\lambda$ as $\re\lambda\to\infty$. 
This implies in particular that $P_2(\lambda)\to0$ as $\lambda\to\infty$ through the reals and that there exists $\gs>0$ such that $\sup_{s\in\R}\abs{P_2(\gs+is)}<\infty$.
Thus by Lemma~\ref{lem:BCStoRLS}  the Coleman--Gurtin part defines 
an impedance passive regular tuple $(A_2,B_2,C_2,D_2)$ with $D_2 = \lim_{\gl\to\infty} P_2(\gl) = 0$. In order to prove the required uniform decay estimate, observe first that $\ell(\lambda)=1+O(|\lambda|^{-1})$ and hence $\lambda\ell(\lambda)=\lambda+O(1)$ and $\lambda/\ell(\lambda)=\lambda+O(1)$ as $|\lambda|\to\infty$ in the right half-plane.
In particular, $|\lambda\ell(\lambda)|\ge \frac12\re\lambda$ for $\re\lambda$ sufficiently large. Moreover,
$$\re\sqrt{\lambda/\ell(\lambda)}={\re{\sqrt{\lambda}}}+O(1)$$ as $ |\lambda|\to\infty$ with $\lambda\in\CC_+,$
and for $\lambda\in\CC_+$ we have $\re{\sqrt{\lambda}}\ge \sqrt{\re{\lambda}}$. It follows that $\re\sqrt{\lambda/\ell(\lambda)}\ge \frac12\sqrt{\re{\lambda}}$ for $\re\lambda$ sufficiently large. Thus
$$|P_2(\lambda)|\le \frac{1}{|\sqrt{\gl\ell(\gl)}|}\bigg|1-\frac{2}{\e^{2\sqrt{\gl/\ell(\gl)}}+1}\bigg|
\le\sqrt{\frac{2}{\re\lambda}}\left(1+\frac{2}{\e^{\sqrt{\re{\lambda}}}-1}\right)$$
when $\re\lambda$ is sufficiently large, and the claim follows.
\end{proof}

\subsection{Proof of Theorem~\ref{thm:SG}}
\label{sec:AppCLOpForm}

Theorem~\ref{thm:SG} is an immediate corollary of the following more detailed proposition.
\begin{proposition}
\label{prp:Ablockstructure}
The operator $\A$ has the form
\begin{align}
  \A &= \pmat{A_{1,-1}&B_1C_{2\Lambda}\\-B_2C_{1\Lambda}&A_{2,-1}-B_2D_1C_{2\Lambda}}, \label{eq:A_block}\\
  \Dom(\A)&= \Setm{\pmat{z_1\\z_2}\in Z_1^{B_1}\times Z_2^{B_2}}{ \begin{matrix}
       A_{1,-1}z_1+B_1 C_{2\Lambda}z_2\in Z_1,\\
   A_{2,-1}z_2
      -B_2 (C_{1\Lambda}z_1+D_1 C_{2\Lambda}z_2) \in Z_2
\end{matrix}
}\label{eq:dom_A},
\end{align}
where $(A_1,B_1,C_1,D_1)$ and $(A_2,B_2,C_2,D_2)$ are the impedance passive 
regular tuples associated to the wave part and the Coleman--Gurtin part, respectively.
The operator $\A$ generates a contraction semigroup on the space $\H$.
\end{proposition}

\begin{proof}
By definition, we have $\H = Z_1\times Z_2$.
Let $(G_1,L_1,K_1)$ and $(G_2,L_2,K_2)$ be the boundary nodes associated to the wave part and the Coleman--Gurtin part, respectively, as defined in Propositions~\ref{prp:RLSWavePart} and~\ref{prp:RLSCGPart}.
If we write $z_1=(u,v)^T$ and $z_2=(w,\eta)^T$,
the operator $\A$ and its domain may be written as
\eq{
  \mathbb{A}\pmat{u\\v\\w\\\eta}
  = \pmat{L_1&0\\0&L_2} \pmat{z_1\\z_2},\qquad 
  \Dom(\mathbb{A}) = \Setm{\pmat{z_1\\z_2}\in \Dom(L_1)\times \Dom(L_2)}{\begin{matrix}G_1z_1=K_2z_2,\\ G_2z_2=-K_1z_1\end{matrix}}.
}
Since $D_2=0$, Lemma~\ref{lem:BCStoRLS} implies that $L_1 = A_{1,-1}+B_1G_1,\ K_1 = C_{1\Lambda}+D_1G_1 $  on $\Dom(L_1)$, and $L_2 = A_{2,-1}+B_2G_2,\ K_2 = C_{2\Lambda}$ on $ \Dom(L_2).$
It follows that, for $(z_1,z_2)^T\in \Dom(\A)$, $G_1z_1=K_2z_2=C_{2\Lambda }z_2$ and $$G_2z_2 = -K_1z_1
= -C_{1\Lambda }z_1-D_1G_1z_1 = -C_{1\Lambda}z_1-D_1C_{2\Lambda}z_2,$$ and hence
\eq{
  L_1 z_1
  &= A_{1,-1} z_1 + B_1G_1z_1
  = A_{1,-1} z_1 + B_1C_{2\Lambda}z_2,\\
  L_2 z_2
  &= A_{2,-1} z_2 + B_2 G_2z_2
  = (A_{2,-1} - B_2D_1C_{2\Lambda}) z_2 - B_2 C_{1\Lambda}z_1 .
}
These formulas show that the operator $\A$ has the desired form  \eqref{eq:A_block} on $\Dom(\A)$ and that the inclusion ``$\subset$'' holds in \eqref{eq:dom_A}.
It therefore remains to show that the inclusion ``$\supset$'' holds in \eqref{eq:dom_A}.
To this end, assume that $(z_1,z_2)^T\in Z_1^{B_1}\times Z_2^{B_2}$
is such that
      $A_{1,-1}z_1+B_1 C_{2\Lambda}z_2\in Z_1$ and
      $A_{2,-1}z_2
      -B_2 (C_{1\Lambda}z_1+D_1 C_{2\Lambda}z_2) \in Z_2$. Then  $(z_1,z_2)^T\in \Dom(L_1)\times \Dom(L_2)$ by Lemma~\ref{lem:BCStoRLS}, so it suffices to show that
$G_1z_1=K_2z_2$ and $G_2z_2 = -K_1z_1$. Lemma~\ref{lem:BCStoRLS} also implies that
 $Z_k\cap \ran(B_k)=\set{0}$ for $k=1,2$.
We have
\eq{
      Z_1\ni A_{1,-1}z_1+B_1 C_{2\Lambda}z_2
      = L_1 z_1 - B_1G_1z_1 + B_1C_{2\Lambda} z_2
      = L_1 z_1 + B_1(-G_1z_1 + K_2 z_2),
}
and since $L_1z_1\in Z_1$
we see that
$G_1z_1=K_2z_2$.
Since $G_1z_1=K_2z_2=C_{2\Lambda}z_2$ and $C_{1\Lambda}z_1+D_1C_{2\Lambda}z_2=K_1z_1$ we find similarly that
\eq{
 Z_2&\ni
 A_{2,-1}z_2 -B_2 (C_{1\Lambda}z_1+D_1 C_{2\Lambda}z_2)
 =L_2z_2  -B_2 (G_2z_2+ K_1z_1) ,
}
which implies that $G_2z_2=-K_1z_1$, as required.

Finally, since
$\A$ has the form
in Proposition~\ref{prp:Ablockstructure} where $(A_k,B_k,C_k,D_k)$ for $k=1,2$  are impedance passive regular tuples with $D_1=1\geq 0$ and $D_2=0$, the operator $\A$ generates a contraction semigroup by~\citel{Pau19}{Lem.~4.2}.
\end{proof}

%%%%%%%%%%%%%%%%%%%%%%%%%%%%%%%%%%%%%%%%%%%%%%%%%%%%%%%%%%%%%%

\section{Resolvent Estimates}
\label{sec:ResolventEstimates}

\noindent We now study the behaviour of the resolvent operator $R(\i s,\A)$ as $s\to\pm\infty$.
In  Section~ \ref{subsec:ResolventEstimates1},
we establish an asymptotic upper bound on $\norm{R(\i s,\A)}$, and then in
Section~\ref{subsec:ResolventEstimates2} we shall  show this upper bound to  be optimal.

\subsection{Upper bound}\label{subsec:ResolventEstimates1}  Our main result here is the following.

\begin{theorem}
  \label{thm:Resolvent}
Assume that~\eqref{condmu} holds.
Then the operator $\A$ satisfies $\gs(\A)\subset \CC_-$ and
  \eq{
    \norm{R(\i s,\A)}=O(\abs{s}^{1/2}), \qquad s\to \pm\infty.
  }
\end{theorem}

%A precise description of $\gs(\A)$, or rather the part of the spectrum lying in a certain neighbourhood of the imaginary axis, will be given in the appendix. 
 The proof of Theorem~\ref{thm:Resolvent} is based on the following abstract result from~\cite{Pau19}, which we state in the special case where $A_1$ has compact resolvent and the eigenvalues of $A_1$ have a uniform gap (but are not necessarily simple).

\begin{theorem}[{\citel{Pau19}{Thm.~3.7}}]
\label{thm:ResolventAbstract}
Let $(A_1,B_1,C_1,D_1)$ and $(A_2,B_2,C_2,D_2)$ be impedance passive 
regular tuples on $Z_1$ and $Z_2$, respectively,  with $\CC^{m\times m}\ni D_1\geq 0$ and $D_2=0\in \CC^{m\times m}$.
Assume 
 that $A_1$ is skew-adjoint with compact resolvent and spectrum $\gs(A_1)=\setm{is_k}{k\in\Z}$, that the eigenvalues of $A_1$ satisfy $\inf_{k\neq l} \abs{s_k-s_l}>0$, and that
 the semigroup generated by $A_2$ is exponentially stable.
In addition, assume 
that there exists a constant $\gamma_0>0$ such that $\norm{C_1z}_{\CC^m}\geq \gg_0 \norm{z}_{Z_1}$ for all $z\in \ker(is_k-A_1)$ and $k\in\N$, and that there exists a non-increasing function
 $\nu:\R_+\to (0,1]$ such that 
\eq{
\re \iprod{ P_2(is)U}{U}_{\CC^m}
\geq \nu(\abs{s}) \norm{U}_{\CC^m}^2, \qquad U\in \CC^m,\ s\in\R.
}
Then the block operator $\A$ defined by
\eq{
  \A &= \pmat{A_{1,-1}&B_1C_{2\Lambda}\\-B_2C_{1\Lambda}&A_{2,-1}-B_2D_1C_{2\Lambda}},\\
  \Dom(\A)&= \Setm{\pmat{z_1\\z_2}\in Z_1^{B_1}\times Z_2^{B_2}}{ \begin{matrix}
       A_{1,-1}z_1+B_1 C_{2\Lambda}z_2\in Z_1,\\
   A_{2,-1}z_2
      -B_2 (C_{1\Lambda}z_1+D_1 C_{2\Lambda}z_2) \in Z_2
\end{matrix}
},
}
 satisfies
$i\R\subset\rho(\A)$ and 
  \eq{
    \norm{R(\i s,\A)}=O(\nu(\abs{s})^{-1}), \qquad s\to \pm\infty.
  }
\end{theorem}

\begin{proof}
The theorem was proved in more general form in~\citel{Pau19}{Thm.~3.7}. 
The regular tuples $(A_1,B_1,C_1,D_1)$ and $(A_2,B_2,C_2,D_2)$ correspond to $(A_c,B_c,C_c,D_c)$ and $(A,B,C,D)$ in~\cite{Pau19}, respectively.
Similarly, the transfer functions $P_1$ and $P_2$ correspond to the transfer functions $G$ and $P$ in~\cite{Pau19}.
The current statement follows from~\citel{Pau19}{Thm.~3.7} if we let $\Omega_\eps=\R$, in which case the 
  required condition~(2) of~\citel{Pau19}{Thm.~3.5} is trivially satisfied.
The non-increasing function $\nu$ corresponds  to the function $\eta$ in~\citel{Pau19}{Thm.~3.7}.
Moreover, as explained in~\citel{Pau19}{Rem.~3.8}, the assumption  that the eigenvalues of $A_1$ satisfy the uniform gap condition $\inf_{k\neq l} \abs{s_k-s_l}>0$ implies that it is possible to choose $\gd$ and $\gamma$ in~\citel{Pau19}{Thm.~3.7} to be constant functions.
With these choices~\citel{Pau19}{Thm.~3.7} implies that $i\R\subset \rho(\A)$ and there exists a constant $M>0$ such that $\norm{R(is,\A)}\leq M /\nu(\abs{s})$ for all $s\in\R$.
\end{proof}

We begin by showing that the semigroup generated by the operator $A_2= L_2\vert_{\ker(G_2)}:D(A_2)\subset Z_2\to Z_2$
introduced in Section~\ref{subsecCG} is exponentially stable whenever condition~\eqref{condmu} is satisfied.

\begin{lemma}
\label{lem:GPBCSEllipticProblem-bis-expstab}
If condition \eqref{condmu} holds, the contraction semigroup generated by $A_2$ is exponentially  stable.
\end{lemma}

\begin{proof}
Recall that $A_2$ is the infinitesimal generator
of a contraction semigroup
on $Z_2=L^2(0,1)\times \M$, as was shown in the proof of Proposition~\ref{prp:RLSCGPart}.
We prove that $i\R\subset \rho(A_2)$ and
$\sup_{r\in\R}\norm{R(ir,A_2)}<\infty$. The claim then follows from the Gearhart-Pr\"uss theorem~\citel{EngNag00book}{Thm.~V.1.11}.
To this end we begin by introducing the space
$\mathcal{N} = L^2_g(0,\infty; H^1_r(0,1))$
of  $H^1_r(0,1)$-valued functions on $(0,\infty)$ which are square-integrable with respect to the measure $g(s)\d s$,
endowed with the inner product
$$\langle\eta,\xi\rangle_\mathcal{N}=\int_0^\infty g(s) \langle\eta(s),\xi(s)\rangle_{H^1_r(0,1)}\, \d s.
$$
By \citel{GatMir08}{Rem.\ 2.3}, condition  \eqref{condmu} is equivalent to the estimate $g(s) \leq \Theta \mu(s)$
for some $\Theta>0$ and all $s>0$. It follows that
\begin{equation}
\label{cont-inclus}
\|\eta\|_\mathcal{N} \leq \sqrt{\Theta} \|\eta\|_\mathcal{M},\qquad \eta \in \mathcal{M},
\end{equation}
and hence $\M \subset \mathcal{N}$ with continuous inclusion.
Next, given $\hat z  = (\hat w, \hat \eta)^T\in Z_2$ and $r \in\R$, we consider the resolvent equation
$(ir - A_2) z = \hat z$
in the unknown $z=(w,\eta)^T\in \Dom(A_2)$.
Taking the inner product of this equation with $z$ in $Z_2$ and using $w(1)=\phi'(0)=0$ and \eqref{Teta}, we get
$$
\|w\|_{H^1_r(0,1)}^2 - \frac12 \int_0^\infty \mu'(s) \|\eta(s)\|_{H^1_r(0,1)}^2\, \d s  =
\re \l (ir - A_2 ) z , z \r_{Z_2} = \re \l \hat z , z \r_{Z_2}.
$$
Since $\mu'(s)\leq0$ for almost all $s>0$, we have
\begin{equation}
\label{w_eq}
\|w\|_{L^2(0,1)}^2  \leq \|w\|_{H^1_r(0,1)}^2 \leq \|z\|_{Z_2} \|\hat z\|_{Z_2}.
\end{equation}
The resolvent equation may be rewritten in component form as
$$\left\{\begin{aligned}
i r w(x) - \phi''(x) &=  \hat w(x),& x\in(0,1),\\
i r \eta(x,s) - T \eta(x,s) - w(x) &= \hat \eta(x,s), \quad&x\in(0,1),\ s>0.
\end{aligned}\right.$$
Recalling that $\M \subset \mathcal{N}$, we may take the inner product in $\mathcal{N}$ of the second equation above with $\eta$.
Taking the real part of the resulting expression, we obtain
\begin{equation}
\label{mix}
-\re \l T\eta, \eta\r_{\mathcal{N}} =  \re \l w, \eta\r_{\mathcal{N}} + \re\l \hat \eta, \eta\r_{\mathcal{N}}.
\end{equation}
Integrating by parts with respect to $s$ and employing a limiting argument
(cf.~\eqref{Teta}) yields
$$
-\re \l T\eta,\eta\r_{\mathcal{N}}
=- \frac12 \int_0^\infty g'(s) \|\eta(s)\|_{H^1_r(0,1)}^2\, \d s
=\frac12 \|\eta\|_\M^2.
$$
Hence,  \eqref{cont-inclus}, \eqref{w_eq} and \eqref{mix} imply that
\begin{align*}
\frac12\|\eta\|_\M^2 \leq
 \sqrt{\Theta}\|w\|_{H^1_r(0,1)}\|\eta\|_\M + \Theta \|\hat \eta\|_\M \|\eta\|_\M
\leq \frac14 \|\eta\|_\M^2 + 2\Theta \|z\|_{Z_2} \|\hat z\|_{Z_2},
\end{align*}
and combining this with \eqref{w_eq} we readily arrive at
\begin{equation}
\label{boundbound}
\|z\|_{Z_2} \leq (1+8\Theta) \|\hat z\|_{Z_2}.
\end{equation}
The desired result now follows at once. Indeed,
since $A_2$ the generator of a contraction semigroup
on $Z_2$, we have $\CC_+\subset \rho(A_2)$. Hence $\sigma(A_2)\cap i\R$ is contained in
the topological boundary of $\sigma(A_2)$, and thus in the approximate point spectrum of
$A_2$. However, \eqref{boundbound} shows that no purely imaginary number can be an approximate
eigenvalue of $A_2$, since otherwise there would exist a sequence of unit vectors $z_n\in D(A_2)$ with
$(\i r -A_2) z_n\to0$ in $Z_2$ as $n\to\infty$, which contradicts \eqref{boundbound}. It follows that $i\R\subset \rho(A_2)$,
and now \eqref{boundbound} yields the bound $\sup_{r\in\R}\norm{R(ir,A_2)}\leq 1+8\Theta$.
\end{proof}

We are now in a position to prove Theorem \ref{thm:Resolvent}.

\begin{proof}[Proof of Theorem~\textup{\ref{thm:Resolvent}}]
By Propositions~\ref{prp:RLSWavePart} and~\ref{prp:RLSCGPart}, 
$(A_k,B_k,C_k,D_k)$ for $k=1,2$ are impedance passive regular tuples with 
$D_1=1\in\CC$ and $D_2=0\in \CC$. Moreover, by Proposition~\ref{prp:RLSWavePart} the operator $A_1$ is skew-adjoint with compact resolvent and spectrum $\gs(A_1)=\setm{ik\pi}{k\in\Z}$ consisting of simple eigenvalues. Furthermore, $\abs{C_1\psi}=\norm{\psi}$ for all $\psi\in \ker(ik\pi -A_1)$ and $k\in\Z$.
By Lemma~\ref{lem:GPBCSEllipticProblem-bis-expstab} the semigroup generated by $A_2$ is exponentially stable. Due to the structure of $\A$ described in Theorem~\ref{thm:SG} we may derive the desired resolvent estimate using Theorem~\ref{thm:ResolventAbstract} provided we can find a non-increasing function
 $\nu:\R_+\to (0,1]$ such that $\re P_2(is)\geq \nu(\abs{s})$ for all $s\in\R$. We shall show that there exists a constant $c_0\in(0,1]$ such that
\eqn{
\label{eq:ReP2lowerbound}
\re P_2(is)\ge\frac{c_0}{1+|s|^{1/2}},\qquad s\in\R.
}
We begin by observing that by exponential stability of the semigroup generated by $A_2$ the transfer function $P_2$ of
 the Coleman--Gurtin part $(A_2,B_2,C_2,D_2)$ extends analytically across the imaginary axis and, in particular, satisfies
$$P_2(is)
= \frac{\tanh\sqrt{is/\ell(is)}}{\sqrt{is \ell(is)}}, \qquad s\neq 0,
$$
where we recall that, for $s\ne0$,
\begin{equation}\label{eq:ell_mu}
\ell(is)=1+\frac{1}{is}\int_0^\infty\mu(r)(1-\e^{-is r})\,\d r.
\end{equation}
Integration by parts yields
\begin{equation}\label{eq:ell_g}
\ell(is)=1+\int_0^\infty g(r)\e^{-isr}\,\d r.
\end{equation}
This expression shows in particular  that we may indeed define $\ell(is)$ and hence $P_2(is)$ in a natural way also for $s=0$, by setting $\ell(0)=2$ and $P_2(0)=1/2$. In particular, both $s\mapsto \ell(is)$ and $s\mapsto P_2(is)$ are continuous on $\R$. We now prove that $\re\ell(is)\ge1$ for all $s\in\R$. Note first that $\re\ell(is)=\re\ell(-is)$ for all $s\in\R$ and, as has just been noted, that $\ell(0)=2$. For $s>0$, we see from \eqref{eq:ell_mu} that
$$\begin{aligned}
\re\ell(is) &=1+\frac1s\int_0^\infty\mu(r)\sin(rs)\,\d r=1+\frac1s\sum_{n=0}^\infty\int_{2n\pi/s}^{2(n+1)\pi/s}\mu(r)\sin(rs)\,\d r\\&=1+\frac1s\sum_{n=0}^\infty\int_{0}^{\pi/s}\left(\mu\left(\frac{2n\pi}{s}+r\right)-\mu\left(\frac{(2n+1)\pi}{s}+r\right)\right)\sin(rs)\,\d r.
\end{aligned}$$
By monotonicity of $\mu$ and non-negativity of $\sin(t)$ for $0\le t\le\pi$ all of the integrands are non-negative, and hence $\re\ell(is)\ge1$ for all $s\in\R$.
Next we prove the asymptotic estimate
\begin{equation}\label{eq:P2_asymp}
P_2(is)=\frac{1\mp i}{\sqrt{2}|s|^{1/2}}+O(|s|^{-3/2}),\qquad s\to\pm\infty.
\end{equation}
Note first that, by integrability of $\mu$, $\ell(is)=1+O(|s|\inv)$ and hence also $\ell(is)\inv=1+O(|s|\inv)$ as $|s|\to\infty$. Thus
\begin{equation}\label{eq:est1}
\sqrt{is/\ell(is)}=\sqrt{is(1+O(|s|\inv))}=|s|^{1/2}\frac{1\pm i}{\sqrt{2}}(1+O(|s|\inv)),\qquad s\to\pm\infty,
\end{equation}
and similarly
\begin{equation}\label{eq:est2}
\frac{1}{\sqrt{is\ell(is)}}=\left(|s|^{1/2}\frac{1\pm i}{\sqrt{2}}(1+O(|s|\inv))\right)\inv=\frac{1\mp i}{\sqrt{2}|s|^{1/2}}(1+O(|s|^{-1}))
\end{equation}
as $s\to\pm\infty$. The estimate \eqref{eq:est1} yields
$$\tanh\sqrt{is/\ell(is)}=1-\frac{2}{\e^{2\sqrt{is/\ell(is)}}+1}=1+O\Big(\e^{-\sqrt{2}|s|^{1/2}}\Big),\qquad |s|\to\infty,$$
and combining this with \eqref{eq:est2} we quickly obtain \eqref{eq:P2_asymp}. It follows that \eqref{eq:ReP2lowerbound} holds for some $c_0\in(0,1]$ and for $|s|$ sufficiently large. Hence by continuity of the map $s\mapsto\re P_2(is)$ on $\R$ it suffices, in order to prove \eqref{eq:ReP2lowerbound}, to show that $\re P_2(is)>0$ for all $s\in \R$. First, from \eqref{eq:ell_g} we see that
$$\mathrm{Im}\,\ell(is)=-\int_0^\infty g(r)\sin(rs)\,\d r,\qquad s\in\R,$$
which implies, in particular, that
$$|\mathrm{Im}\,\ell(is)|<\int_0^\infty g(r)\,\d r=1,\qquad s\in\R.$$
Let us denote by $\Sigma$ the sector $\{z\in\mathbb{C}\setminus\{0\}:|\arg z|<\pi/4\}$. Since $\re{\ell(is)}\ge1$ and $\abs{\im{\ell(s)}}<1$ we see that $\ell(is)\in\Sigma$ for all $s\in\R$, and because $\Sigma$ is invariant under the inversion $z\mapsto z\inv$ we also have $\ell(is)\inv\in\Sigma$ for all $s\in\R$.  Fix $s>0$ and let $\theta=\arg\ell(is)$. Here and in what follows we take $\arg $ to be the principal value of the argument, so that $|\arg \gl|\le\pi$ for all $\gl\in\CC$. Then $\arg\sqrt{is/\ell(is)}=\frac\pi4-\frac\theta2$ and  $\arg\sqrt{is\ell(is)}=\frac\pi4+\frac\theta2$. Let $a,b>0$ be such that $\sqrt{is/\ell(is)}=a+ib$. Then
$$\tanh\sqrt{is/\ell(is)}=\frac{\sinh(2a)+i\sin(2b)}{\cosh(2a)+\cos(2b)}.$$
Using that $|\sin(2b)|<2b$ and $\sinh(2a)>2a$ together with monotonicity of the arctangent, we find that
$$\big|\arg\tanh\sqrt{is/\ell(is)}\big|=\left|\tan\inv\left(\frac{\sin(2b)}{\sinh(2a)}\right)\right|< \tan\inv\left(\frac{b}{a}\right)=\arg \sqrt{is/\ell(is)}=\frac\pi4-\frac\theta2.$$
Since $\arg P_2(is)=\arg \tanh\sqrt{is/\ell(is)}-\arg\sqrt{is\ell(is)}$, we obtain $-\pi/2<\arg P_2(is)<\pi/4$. In particular, $\re P_2(is)>0$. An analogous argument applies when $s<0$, and thus there exists $c_0\in(0,1]$ such that~\eqref{eq:ReP2lowerbound} holds.
Hence if we let
 $\nu(r) = c_0/(1+\sqrt{r})$ for $r\geq0$, then~\eqref{eq:ReP2lowerbound} yields 
$\re P_2(is)\geq \nu(\abs{s})$ for all $s\in\R$.
It follows from Theorem~\ref{thm:ResolventAbstract} that  $ \gs(\A)\subset \CC_-$ and $\norm{R(is,\A)} = O(\abs{s}^{1/2})$ as $s\to\pm \infty$, as required.
 \end{proof}

\subsection{Optimality}\label{subsec:ResolventEstimates2}

\noindent
The following result shows that the resolvent estimate
in Theorem \ref{thm:Resolvent} is optimal.

\begin{theorem}
\label{stima basso}
Suppose that $\gs(\A)\subset\CC_-$.
Then
$$
\limsup_{s \to \infty} s^{-1/2}\|R(\i s,\A)\|>0.
$$
\end{theorem}

\begin{proof}
For $n\ge1$ let
$\hat z_n = ( \hat u_n, \hat v_n, 0, 0 )^T \in \H$, where
$\hat u_n(x) = \sin (2 \pi n x)/2 \pi n$ and $\hat v_n(x)= \cos (2 \pi n  x)$.
In particular,  $\|\hat z_n\|_\H =1$ for all $n\ge1$.
Since $\gs(\A)\subset\CC_-$ by assumption, the equation
$(2 \pi n \i - \A) z_n = \hat z_n$
has a unique solution
$z_n = (u_n,v_n,w_n,\eta_n) ^T\in \Dom(\A)$ for each $n\ge1$.
The components satisfy the system
\begin{equation}\label{R1}
\left\{\begin{aligned}
2 \pi n \i\, u_n (x)- v_n(x) &= \hat u_n(x), &x\in(-1,0),\\
2 \pi n \i\, v_n(x) - u_{n}''(x) &= \hat v_n(x), &x\in(-1,0),\\
 2 \pi n \i\, w_n(x) -  \phi_{n}''(x) &= 0,&x\in(0,1),\\
 2 \pi n \i\, \eta_n(x,s) - T \eta_n(x,s) - w_n(x)&=0,&x\in(0,1),\ s>0,
\end{aligned}\right.
\end{equation}
where, as before,
$$
\phi_n (x)= w_n(x) + \int_0^\infty \mu(s) \eta_n(x,s) \,\d s,\qquad x\in(0,1).
$$
Let us introduce the auxiliary functions
$U_n^{\pm} = \frac12 (v_n \pm u_{n}')$ on $(-1,0)$. Then, using the first two equations in~\eqref{R1},
it is readily seen that
\begin{align*}
(U_{n}^{+})'(x) =  2 \pi n\i \, U_n^{+}(x) - \cos (2 \pi n  x),\qquad
(U_{n}^{-})'(x) =  -2 \pi n\i \, U_n^{-}(x)
\end{align*}
for all $x\in(-1,0),$
and solving these subject to  $U_n^{+}(-1) + U_n^{-}(-1)=v_n(-1)=0$ yields
\begin{align*}
 U_n^{+}(x) &= \e^{2 \pi n \i (x+1)} U_n^{+}(-1) -\int_{-1}^x \e^{2 \pi n \i (x-\tau)}\cos (2 \pi n  \tau)\, \d \tau,\\
 U_n^{-}(x) &= -\e^{-2 \pi n \i (x+1)} U_n^{+}(-1)
\end{align*}
for all $x\in(-1,0).$ Since $v_n=U_n^++U_n^-$ and $u_n'=U_n^+-U_n^-$, it follows that
\begin{align}
\label{vn}
v_n(x) 
&= \big(\e^{2 \pi n \i (x+1)} -\e^{-2 \pi n \i (x+1)}\big) U_n^{+}(-1)
-\int_{-1}^x \e^{2 \pi n \i (x-\tau)} \cos (2 \pi n  \tau) \,\d \tau,\\
\label{un}
u_{n}'(x) &= \big(\e^{2 \pi n \i (x+1)} +\e^{-2 \pi n \i (x+1)}\big) U_n^{+}(-1)
-\int_{-1}^x \e^{2 \pi n \i (x-\tau)} \cos (2 \pi n  \tau) \,\d \tau
\end{align}
for all $x\in(-1,0)$. In particular, \eqref{vn} yields
\begin{equation}\label{boundvn}
\begin{aligned}
\|v_n\|_{L^2(-1,0)}^2 &\geq \frac12
\int_{-1}^0 \big|\big(\e^{2 \pi n \i (x+1)} -\e^{-2 \pi n \i (x+1)}\big) U_n^{+}(-1)\big|^2\, \d x \\
&\qquad- \int_{-1}^0 \bigg|\int_{-1}^x \e^{2 \pi n \i (x-\tau)} \cos (2 \pi n  \tau) \,\d \tau\bigg|^2\, \d x\geq |U_n^+(-1)|^2-\frac13.
\end{aligned}
\end{equation}
It moreover follows from \eqref{vn} and \eqref{un} that $v_n(0)=-1/2$ and $u_{n}'(0) =2U_n^{+}(-1)-1/2$.
Integrating the fourth equation in \eqref{R1} and using the fact that $\eta_n(x,0)=0$ yields
$$
\eta_n(x,s) = \frac{1-\e^{-2\pi n \i s}}{2 \pi n \i} w_n(x),\qquad x\in(0,1),\ s>0.
$$
Hence $\phi_n = \alpha_n w_n$, where
\begin{equation}
\label{phin}
\alpha_n = 1 + \frac{1}{2 \pi n \i}\left(\kappa - \int_0^\infty \mu(s)\e^{-2\pi n \i s} \,\d s\right)
\end{equation}
with $\kappa=\int_0^\infty\mu(s)\,\d s$, and in particular $\phi_{n}'(0)= \alpha_n w_{n}'(0).$
Note also that $\alpha_n \to 1$ as $n\to\infty$, so by considering only sufficiently large values of $n\ge1$ we may assume that $\alpha_n\ne0$.
Now using \eqref{phin} in the third equation in \eqref{R1} we find that
$$
w_{n}'(x) = \frac{2 \pi n \i}{\alpha_n} \int_0^x w_n(\tau) \,\d \tau + w_{n}'(0),\qquad x\in(0,1).
$$
Let us set
$\sigma_n = \sqrt{2 \pi n \i/\alpha_n}$
and 
$$W_n^{\pm}(x) = \frac12 \left(w_n(x) \mp \sigma_n\int_0^x w_n(\tau) \,\d \tau\right),\qquad x\in(0,1).$$ 
Then
$(W_{n}^{\pm})' =\frac12 w_{n}'(0) \mp \sigma_n W_n^{\pm}$, and solving these differential equations  subject to $W_n^{+}(1) + W_n^{-}(1)=w_n(1)=0$ yields
\begin{align*}
 W_n^{+}(x) &= \e^{-\sigma_n(x-1)} W_n^+(1) + \frac{1-\e^{-\sigma_n(x-1)}}{2 \sigma_n}w_{n}'(0),\\
 W_n^{-}(x) &  = -\e^{\sigma_n(x-1)} W_n^+(1) - \frac{1-\e^{\sigma_n(x-1)}}{2 \sigma_n}w_{n}'(0)
\end{align*}
for all $x\in(0,1)$. Since $W_n^{+}(0) = W_n^{-}(0)$, it follows that
$$
W_n^+(1)=\frac{\e^{\sigma_n}+ \e^{-\sigma_n} -2}{2(\e^{\sigma_n}+\e^{-\sigma_n})\sigma_n} w_{n}'(0).
$$
Thus, using the relation $w_n(0)=W_n^{+}(0) + W_n^{-}(0)$ we see, after a few elementary manipulations, that
\begin{equation}
\label{wnzero}
w_n(0)= - \frac{\tanh \sigma_n}{\sigma_n}\, w_{n}'(0).
\end{equation}
Observe in particular that $\tanh \sigma_n\to 1$
as $n\to\infty$.
Combining the coupling conditions $v_n(0)=w_n(0)$, $u_{n}'(0)=\phi_{n}'(0)$ with the identities $v_n(0)=-1/2$, $u_{n}'(0) =2U_n^{+}(-1)-1/2$ obtained above and using    the fact that $\phi_{n}'(0)= \alpha_n w_{n}'(0)$, it follows from \eqref{wnzero} that
$$\left\{\begin{aligned}
2 w_{n}'(0) \tanh \sigma_n  &= \sigma_n,\\
2w_{n}'(0) \alpha_n &=  4 U_n^+(-1)-1.
\end{aligned}\right.
$$
Now the definition of $\sigma_n$ implies that
$$
U_n^+(-1) = \frac14 + \frac{\alpha_n \sigma_n}{4\tanh \sigma_n}
\sim \frac{\sigma_n}{4}\sim \frac{\sqrt{2 \pi n \i}}{4}
$$
as $n\to\infty$, and hence, using \eqref{boundvn}, 
$$
\|z_n\|_{\H}^2 \geq \|v_n\|_{L^2(-1,0)}^2\geq |U_n^+(-1)|^2  -\frac13\geq \frac{\pi n}{16} 
$$
for all sufficiently large $n\ge1$.
The result now follows from the fact that $\|\hat z_n\|_{\H}=1$ for all $n\ge1$.
\end{proof}

\begin{remark}
An alternative approach to proving optimality of the resolvent bound in Theorem~\ref{thm:Resolvent} is to give a precise description of the part of the spectrum of $\A$ lying in a neighbourhood of the imaginary axis and then to bound the resolvent norm from below by means of the elementary estimate $\|R(\lambda,\A)\|\ge\dist(\lambda,\gs(\A))^{-1}$ for $\gl\in\rho(\A)$. Our approach is shorter and more direct. The required description of the spectrum of $\A$ may nevertheless be found in the appendix.
\end{remark}

%%%%%%%%%%%%%%%%%%%%%%%%%%%%%%%%%%%%%%%%%%%%%%%%%
\section{Energy Decay}

\noindent
In this last main section we convert the resolvent estimate obtained in Theorem~\ref{thm:Resolvent} into a decay rate for the semigroup $\St$ generated by $\A$. In particular, we shall show that $\St$ is semi-uniformly polynomially stable. The key to this is the following well-known theoretical result due to Borichev and Tomilov \cite[Thm.~2.4]{BorTom10}.

\begin{theorem}\label{thm:BT}
Let $A$ be the generator of a bounded $C_0$-semigroup $\Tt$ on a Hilbert space $Z$, and suppose that $\sigma(A)\cap i\R=\emptyset$. For each fixed $\alpha>0$ the following statements are equivalent:
\begin{enumerate}
\item[\textup{(i)}] $\|R(is,A)\|=O(|s|^\alpha)$ as $|s|\to\infty$;
\item[\textup{(ii)}] $\|T(t)A^{-1}\|=O(t^{-1/\alpha})$ as $t\to\infty$;
\item[\textup{(iii)}]$\|T(t)z\|_Z=o(t^{-1/\alpha})$ as $t\to\infty$ for every $z\in D(A)$.
\end{enumerate}
\end{theorem}

A $C_0$-semigroup satisfying the equivalent conditions of Theorem~\ref{thm:BT} is said to be \emph{polynomially stable} (\emph{with parameter $\alpha$.}) By \cite[Prop.~1.3]{BatDuy08} the implication (ii)$\implies$(i) holds much more generally and even for $C_0$-semigroups on Banach spaces, whereas passing from (i) to (ii) in general requires a logarithmic correction factor in the Banach space setting, as is shown in  \cite[Thm.~1.5]{BatDuy08} and \cite[Thm.~4.1]{BorTom10}. For $C_0$-semigroups on Hilbert spaces the implication (i)$\implies$(ii) has recently been extended beyond the case of polynomial resolvent growth in \cite[Thm.~3.2]{RozSei19}.

From now on we consider the $C_0$-semigroup $\St$ generated by the operator $\A$ associated with system \eqref{WGP}. Since the orbits of the semigroup $\St$ with initial values $z_0\in D(\A)$ correspond to classical solutions of the abstract Cauchy problem \eqref{eq:ACP}, we may interpret parts (ii) and (iii) of Theorem~\ref{thm:BT} as statements about (uniform) rates of energy decay of classical solutions to our problem \eqref{WGP}.

\begin{theorem}\label{thm:decay}
Assume that~\eqref{condmu} holds. Then the semigroup $\St$ generated by the operator $\A$ is polynomially stable with parameter $1/2$. In particular, for any vector $z_0=(u,v,w,\eta)^T\in D(\A)$, the associated classical solution of \eqref{eq:ACP} satisfies $\|S(t)z_0\|_\H=o(t^{-2})$ as $t\to\infty$.
\end{theorem}

\begin{proof}
This follows immediately from Theorems~\ref{thm:Resolvent} and \ref{thm:BT}.
\end{proof}

Our next result shows that optimality of the resolvent bound in Theorem~\ref{thm:Resolvent}, as established in Theorem~\ref{stima basso}, implies optimality of the decay rate in Theorem~\ref{thm:decay}.

\begin{proposition}\label{prp:opt}
Let $\St$ be the $C_0$-semigroup generated by $\A$. Given any function $r:\R_+\to(0,\infty)$ such that $r(t)=o(t^{-2})$ as $t\to\infty$, there exists a vector $z_0=(u,v,w,\eta)^T \in D(\A)$ such that
\begin{equation}\label{eq:limsup}
\limsup_{t\to\infty}\frac{\|S(t)z_0\|_\H}{r(t)}=\infty.
\end{equation}
In other words, for any such function $r$ there exist initial data giving rise to a classical solution of \eqref{eq:ACP} whose energy decays strictly more slowly than $r(t)$ as $t\to\infty$.
\end{proposition}

\begin{proof}
Replacing $r(t)$ by $\sup_{s\ge t}r(s)$ for $t\ge0$ if necessary, we may assume that $r$ is non-increasing. Suppose, for the sake of a contradiction, that \eqref{eq:limsup} is false for all $z_0\in D(\A)$. Since $1\in\rho(\A)$ by contractivity of the semigroup $\St$ and since $(I-\A)^{-1}$ maps $\H$ onto $D(\A)$, we then have
$\sup_{t\ge0}r(t)^{-1}\|S(t)(I-\A)^{-1}z_0\|_\H<\infty$
for all $z_0\in \H$. Thus $\sup_{t\ge0}r(t)^{-1}\|S(t)(I-\A)^{-1}\|<\infty$ by the uniform boundedness principle, and we may let $C=\sup_{t\ge0}r(t)^{-1}\|S(t)(I-\A)^{-1}\|$, a positive real number. 
Note in particular that $\|S(t)(I-\A)^{-1}\|\to0$ as $t\to\infty$, and hence $\sigma(\A)\subset\CC_-$ by \cite[Prop.~1.3]{BatDuy08}.
 On the other hand, it follows straightforwardly from Theorem~\ref{stima basso} and \cite[Prop.~5.4]{CPSST19} that
$\limsup_{t\to\infty} t^{2}\|S(t)(I-\A)^{-1}\|>0$, so we may find a sequence $(t_n)_{n\ge1}$ of positive real numbers such that $t_n\to\infty$ as $n\to\infty$ and a constant $c>0$ such that $t_n^2\|S(t_n)(I-\A)^{-1}\|\ge c$ for all $n\ge1$.  Thus $Cr(t_n)\ge  \|S(t_n)(I-\A)^{-1}\|\ge c t_n^{-2}$ for all $n\ge1$, which contradicts the assumption that $r(t)=o(t^{-2})$ as $t\to\infty$.
\end{proof}

%%%%%%%%%%%%%%%%%%%%%%%%%%%%%%%%%%%%%%%%%%%%%%%%%
\section*{Appendix:\ The Spectrum of $\A$}

\theoremstyle{plain}
\newtheorem{theoremAPP}{Theorem}[section]
\newtheorem{lemmaAPP}[theoremAPP]{Lemma}
\renewcommand{\thetheoremAPP}{A.\arabic{theoremAPP}}
\renewcommand{\thelemmaAPP}{A.\arabic{lemmaAPP}}
\renewcommand{\theequation}{A.\arabic{equation}}
\setcounter{equation}{0}

\noindent
In this appendix we describe, using similar techniques as in \cite{Del18}, the spectrum of $\A$ near the imaginary axis. We shall  assume throughout that~\eqref{condmu} holds. For $\delta>0$  as in \eqref{condmu}, we introduce the vertical strip
$\Pi_\delta = \{ \gl \in \CC : -\tfrac{\delta}{2} < \re \gl \leq0 \}$, and
we denote by $Z_\ell$ the
zero set of the map $\ell:\Pi_\gd\setminus \set{0}\to\CC$, noting that, by~\eqref{condmu},
$$
\ell(\gl) = 1 + \frac{1}{\gl} \int_0^\infty \mu(s)(1-\e^{-\gl s})\, \d s
$$
is indeed well-defined for every $\gl \in \Pi_\delta \setminus \set{0}$.
We also consider the set
$$
\Sigma=\Setm{\gl \in \Pi_\delta \setminus (Z_\ell \cup \{ 0\} )
}{ \sqrt{\ell(\gl)\gl} \sinh\gl \cosh \sqrt{\tfrac{\gl}{\ell(\gl)}}+
\cosh\gl \sinh \sqrt{\tfrac{\gl}{\ell(\gl)}}=0  }.
$$

\begin{theoremAPP}
\label{teospec}
The spectrum $\sigma(\A)$ of the operator
$\A$ satisfies $
\sigma(\A)  \cap \Pi_\delta = \Sigma \cup Z_\ell .
$
\end{theoremAPP}

In the proof of this theorem we shall make use of the following technical lemma whose proof is similar to the argument in the first part of the proof of Lemma~\ref{lem:GPBCSEllipticProblem-bis} and consequently omitted.

\begin{lemmaAPP}
\label{lemma0}
For any $\hat \eta \in \M$ and $\gl \in \Pi_\delta$, the function $\xi_{\hat\eta,\gl}$ defined by
$$
\hat \xi_{\hat\eta,\gl}(x,s)=
\int_0^s \e^{-\gl(s-\sigma)} \hat \eta(x,\sigma) \,\d \sigma,\qquad x \in (0,1),\ s>0,
$$
belongs to $\M$, and 
$
\|\hat \xi_{\hat\eta,\gl}\|_\M \leq \frac{2\sqrt{C}}{ \delta + 2 \re \gl}\|\hat \eta\|_\M.
$
\end{lemmaAPP}

\begin{proof}[Proof of Theorem~\ref{teospec}]

We divide the proof into three steps.

\medskip
\noindent
{\it Step\ 1.} We first show that $0 \not \in \sigma(\A)$, which is to say that for every
$\hat z =(\hat u,\hat v,\hat w,\hat \eta)^T \in \H$ the equation
$\A z =\hat z$ has a unique solution $z=(u,v,w,\eta)^T\in D(\A)$. Componentwise, we obtain 
$$
\left\{\begin{aligned}
 v (x)&= \hat u(x),&x\in(-1,0),\\
 u''(x)&= \hat v(x),&x\in(-1,0),\\
 \phi''(x)&= \hat w(x),&x\in(0,1),\\
 T\eta (x,s)+ w(x)& = \hat \eta(x,s), \quad&x\in(0,1),\ s>0.
\end{aligned}\right.
$$
Integrating the last equation and using $\eta(x,0)=0$, we obtain
\begin{equation}\label{SS4}
 \eta(x,s)= s w(x) - \hat \xi_{\hat \eta, 0}(x,s),\qquad x\in(0,1),\ s>0.
 \end{equation}
 Solving for $u$ and $\phi$ by using the conditions $u(-1)=0$ and $\phi(1)=0$ we find that 
\begin{align*}
u(x) &= a (x+1) + x \int_{-1}^x \hat v(r) \,\d r - \int_{-1}^x r \hat v(r) \,\d r,  &x\in(-1,0),\\\phi(x)& = b (1-x) - x \int_{x}^1\hat w(r) \,\d r +  \int_{x}^1 r \hat w(r) \,\d r, & x\in(0,1),
\end{align*}
for some $a,b\in\CC$. Note in particular that $u \in H^2(-1,0)$ and $\phi \in H^2(0,1)$.
Next, recalling the definition of $\phi$, we find
$$w(x) = \frac12 \phi(x) +  \frac12 \int_0^\infty \mu(s) \hat \xi_{\hat \eta, 0}(x,s) \,\d s,\qquad x\in(0,1).$$
From Lemma \ref{lemma0} we obtain
$$\left\| \int_0^\infty \mu(s) \hat \xi_{\hat \eta, 0}(x,s) \,\d s \right\|_{H^1_r(0,1)}
\leq \sqrt{\kappa}\|\hat \xi_{\hat \eta, 0}\|_\M \leq 2\delta^{-1} \sqrt{\kappa C}\|\hat\eta\|_\M,$$
where $\kappa=\int_0^\infty\mu(s)\,\d s$. In particular, we have $w\in H^1_r(0,1)$.
Since the map $s\mapsto s^2\mu(s)$ is an element of $L^1(0,\infty)$, it follows from Lemma~\ref{lemma0} that the function $\eta$ defined in \eqref{SS4}
belongs to $\M$. Note also that
$\eta_s = w -\hat \eta \in \M$.
Finally, by monotonicity of $\mu$, we have
\begin{align*}
\|\eta(s)\|_{H^1_r(0,1)} 
\leq s\|w\|_{H^1_r(0,1)} + \frac{1}{\mu(s)} \bigg(\int_0^s \mu(\sigma)\,\d \sigma\bigg)^{1/2} \|\hat \eta\|_\M \to 0,\qquad s\to0,
\end{align*}
which implies that $\eta\in D(T)$.
It remains only to show that the constants $a,b\in\CC$ may be chosen in such a way that the coupling conditions $v(0)=w(0)$ and  $u'(0) = \phi'(0)$ are satisfied.
Straightforward computations show that these conditions are equivalent to
\begin{align*}
b  = 2 \hat u(0) - \int_0^\infty \mu(s)\hat \xi_{\hat \eta, 0}(0,s)\,\d s - \int_{0}^1 r \hat w(r) \,\d r,\quad\,\,\,
a = -b - \int_{-1}^0 \hat v (r) \,\d r  - \int_0^1 \hat w(r) \,\d r.
\end{align*}
Since all of the integrals are finite, we may indeed find suitable constants $a,b \in\CC$.

\medskip
\noindent
{\it Step\ 2.} We prove that
$Z_\ell \subset \sigma(\A)$ by showing that $\gl - \A$ is not onto for $\gl \in Z_\ell$.
Pick
any $\hat w \in L^2(0,1) \setminus H^1(0,1)$ and set
$\hat z = (0,0,\hat w,0)^T \in \H$.
If $\gl - \A$ were onto, then there would exist $z=(u,v,w,\eta)^T\in D(\A)$ such that
$\gl z - \A z =\hat z$.
In component form, the problem becomes
$$
\left\{\begin{aligned}
\gl u (x)&= v(x),&x\in(-1,0),\\
\gl v(x) &= u''(x),&x\in(-1,0),\\
\gl w(x) - \phi''(x) &= \hat w(x), &x\in(0,1),\\
\gl \eta(x,s) -  T\eta(x,s)& = w(x),\quad &x\in(0,1),\ s>0.
\end{aligned}\right.
$$
Integrating the last equation with $\eta(x,0)=0$ we get
$
\eta(x,s)= \frac{1}{\gl} (1- \e^{-\gl s}) w(x)
$ for $x\in(0,1)$ and $s>0$.
Since $\ell(\gl)=0$ a short calculation yields
$\phi = 0$, and now
the third equation implies that
$w = \gl^{-1}\hat w \notin H^1_r(0,1)$. This is the desired contradiction.

\medskip
\noindent
{\it Step\ 3.} 
Let $\gl \in \Pi_\delta \setminus ( Z_\ell \cup \{ 0\} )$ be arbitrary. In the light of Steps 1 and 2,
the result will be proved once we have shown that
$\gl \in \Sigma \iff\gl \in \sigma(\A)$. To this end, 
let us fix an arbitrary $\hat z =(\hat u,\hat v,\hat w,\hat \eta)^T \in \H$.
Our goal is to show that the equation
$\gl z - \A z =\hat z$
admits a unique solution $z=(u,v,w,\eta)^T\in D(\A)$ if and only if
\begin{equation}
\label{claim}
\sqrt{\ell(\gl)\gl} \sinh\gl \cosh \sqrt{\tfrac{\gl}{\ell(\gl)}}+
\cosh\gl \sinh \sqrt{\tfrac{\gl}{\ell(\gl)}} \neq0.
\end{equation}
In component form, our problem becomes
$$
\left\{\begin{aligned}
\gl u (x)-v(x)&= \hat u(x),&x\in(-1,0),\\
\gl v(x) -u''(x)&=\hat v(x) ,&x\in(-1,0),\\
\gl w(x) - \phi''(x) &= \hat w(x), &x\in(0,1),\\
\gl \eta(x,s) -  T\eta(x,s)-w(x)& = \hat\eta(x,s),\quad &x\in(0,1),\ s>0.
\end{aligned}\right.
$$
Integrating the last equation and using $\eta(x,0)=0$ we find 
\begin{equation}
\label{etaeta}
\eta(x,s)= \frac{1- \e^{-\gl s}}{\gl}  w(x) + \hat \xi_{\hat \eta ,\gl}(x,s),\qquad x\in(0,1),\ s>0.
\end{equation}
Recalling the definition of $\phi$, we have
\begin{equation}
\label{wx}
w(x) = \frac{\phi(x)}{\ell(\gl)}- \frac{1}{\ell(\gl)} \int_0^\infty \mu(s)\hat \xi_{\hat \eta ,\gl}(x,s) \,\d s,\qquad x\in(0,1).
\end{equation}
Using the boundary condition $u(-1)=0$ we obtain 
$$u(x) = a(\gl)\sinh (\gl(x+1)) -U(\gl,x),\qquad x\in(-1,0),$$
where $a(\gl)\in \CC$ and
$$
U(\gl,x)= \frac{1}{\gl} \int_{-1}^x \sinh(\gl(x-r))(\hat v(r)+\gl \hat u(r)) \,\d r,\qquad x\in(-1,0).
$$
Once $u$ has been found, $v$ is determined by the first equation of our system.
It is straightforward to check that $u \in H^2(-1,0)$ and $v \in H^1_l(-1,0)$. Let us introduce the auxiliary function
$$
\hat \varrho_{\hat \eta ,\gl} (x)=\frac{\gl}{\ell(\gl)}\int_0^\infty \mu(s)\hat \xi_{\hat \eta ,\gl}(x,s) \,\d s,\qquad x\in(0,1).
$$
The general solution for $\phi$ subject to the boundary condition $\phi(1)=0$ may be written as
\begin{equation}
\label{phix}
\phi(x) = - b(\gl)\sinh \Big(\sqrt{\tfrac{\gl}{\ell(\gl)}}(1-x)\Big) -\Phi(\gl,x),\qquad x\in(0,1),
\end{equation}
where $b(\gl)\in \CC$ and
$$
\Phi(\gl,x)= \sqrt{\frac{\ell(\gl)}{\gl}} \int_{x}^1 \sinh \Big(\sqrt{\tfrac{\gl}{\ell(\gl)}}(r-x)\Big)
(\hat w(r) +\hat \varrho_{\hat \eta ,\gl} (r)) \,\d r,\qquad x\in(0,1).
$$
By \eqref{wx}, we also have
\begin{equation}
\label{explicitw}
w(x)= - \frac{b(\gl)}{\ell(\gl)}\sinh \Big(\sqrt{\tfrac{\gl}{\ell(\gl)}}(1-x)\Big) -\frac{\Phi(\gl,x)}{\ell(\gl)}
- \frac{\hat \varrho_{\hat \eta ,\gl}(x)}{\gl},\qquad x\in(0,1).
\end{equation}
Once $w$ has been found, $\eta$ is determined by~\eqref{etaeta}.
We now show that $\phi\in H^2(0,1)$, $w \in H^1_r(0,1)$ and $\eta \in D(T)$.
In fact, it follows from \eqref{phix} and \eqref{explicitw} that $\phi\in H^2(0,1)$
and $w \in H^1_r(0,1)$ provided that
$\hat \varrho_{\hat \eta ,\gl} \in H^1_r(0,1)$. 
The latter follows from Lemma \ref{lemma0}, which also yield the bounds
$$
\left\|\frac{\gl}{\ell(\gl)}\int_0^\infty \mu(s)\hat \xi_{\hat \eta ,\gl}(x,s) \,\d s \right\|_{H^1_r(0,1)}
\leq \bigg|\frac{\gl}{\ell(\gl)}\bigg| \sqrt{\kappa}\|\hat \xi_{\hat \eta ,\gl}\|_\M
\leq \bigg|\frac{\gl}{\ell(\gl)}\bigg| \frac{2 \sqrt{\kappa C}}{\delta+2\re\gl}\|\hat\eta\|_\M.
$$
In order to prove that $\eta\in D(T)$, we first
show that $\eta \in \M$. Since $\hat \xi_{\hat \eta ,\gl}\in \M$ by Lemma~\ref{lemma0}, we only need to show
that the map $s\mapsto(1- \e^{-\gl s}) w$ lies in $\M$. To this end, note that
$$
\int_0^\infty \mu(s) |1- \e^{-\gl s}|^2 \|w\|_{H^1_r(0,1)}^2 \,\d s \leq 2\kappa \|w\|_{H^1_r(0,1)}^2+
\|w\|_{H^1_r(0,1)}^2 \int_0^\infty \mu(s) \e^{-2 (\re \gl) s} \,\d s.
$$
Since $-2 \re \gl <\delta$, \eqref{condmu} implies that 
$\int_0^\infty \mu(s) \e^{-2 (\re \gl) s} \,\d s <\infty$,
and hence $\eta \in \M$. Thus
$\eta_s  = \hat \eta + w - \gl \eta \in \M$
as well.
Finally, by monotonicity of $\mu$, we have
$$
\|\eta(s)\|_{H^1_r(0,1)} \leq \frac{|1-\e^{-\gl s}|}{|\gl|}\|w\|_{H^1_r(0,1)}
+ \frac{\e^{-(\re \gl) s}}{\mu(s)} \bigg(\int_0^s \mu(\sigma)\,\d \sigma\bigg)^{1/2} \|\hat \eta\|_\M \to 0
$$
as $s\to 0$. Thus $\eta\in D(T)$.
It remains only to show that the coefficients $a(\gl), b(\gl)\in\CC$ may be chosen in such a way that the coupling conditions $v(0)=w(0)$ and  $u'(0) = \phi'(0)$ are satisfied. 
It is straightforward to show that these conditions are equivalent to the matrix equation
$$
\begin{pmatrix}
\ell(\gl)\gl \sinh\gl & \sinh \sqrt{\tfrac{\gl}{\ell(\gl)}}  \\
\ell(\gl)\gl  \cosh\gl & -\sqrt{\ell(\gl) \gl}\cosh\sqrt{\tfrac{\gl}{\ell(\gl)}}
\end{pmatrix}
\begin{pmatrix}
a(\gl) \\
b(\gl)
\end{pmatrix}=
\begin{pmatrix}
\hat f(\gl) \\
\hat g(\gl)
\end{pmatrix},
$$
where
$$
\hat f(\gl)=  \ell(\gl)\bigg(\hat u(0) + \gl U(\gl,0) - \frac{\hat \varrho_{\hat \eta ,\gl}(0)}{\gl}\bigg)
-\Phi(\gl,0),\quad\,\, \hat g(\gl) = \ell(\gl)(U_x(\gl,0) - \Phi_x(\gl,0)).
$$
Hence we may uniquely determine $a(\gl)$,  $b(\gl)$
if and only if the determinant of the matrix appearing on the left-hand side
is non-zero, which in turn is equivalent to  \eqref{claim}.
\end{proof}

%%%%%%%%%%%%%%%%%%%%%%%%%%%%%%%%%%%%%%%%%%%%%%%%%
\begin{Acknowledgments}
We thank Vittorino Pata for bringing to our attention the model studied in the present paper. 
\end{Acknowledgments}
%%%%%%%%%%%%%%%%%%%%%%%%%%%%%%%%%%%%%%%%%%%%%%%%%

\bibliographystyle{plain}

\end{document}